\documentclass[oupthm,crop,info]{CUP-JNL-BCM}%

\usepackage{graphicx}
\usepackage{multicol,multirow}
\usepackage{amsmath,amssymb,amsfonts}
\usepackage{mathrsfs}
\usepackage{amsthm}
\usepackage{rotating}
\usepackage{appendix}
\usepackage[numbers]{natbib}
\usepackage{pstricks-add}
\usepackage[normalem]{ulem}

\theoremstyle{oupplain}
\newtheorem{theorem}{Theorem}[section]
\newtheorem{maintheorem}{Theorem}
\newtheorem{lemma}[theorem]{Lemma}

\theoremstyle{oupdefinition}

\theoremstyle{oupremark}
\newtheorem{remark}[theorem]{Remark}
\newtheorem{example}[theorem]{Example}
\theoremstyle{oupproof}
\newtheorem{proof}{Proof}
\newtheorem{altproof}{Proof of Theorem \ref{thm:Gkl(n,f)spine} }

\newcommand{\Z}{\mathbb{Z}}
\newcommand{\pres}[2]{\langle {#1}\ |\ {#2} \rangle}
\allowdisplaybreaks

\articletype{RESEARCH ARTICLE}
\jname{Canadian Mathematical Society}
\jyear{2020}

\begin{document}

\begin{Frontmatter}

\title[3-manifold spine cyclic presentations with seldom seen Whitehead graphs]{3-manifold spine cyclic presentations with seldom seen Whitehead graphs}

\author{Gerald Williams}

\authormark{Gerald Williams}

\address{\orgname{School of Mathematics, Statistics and Acturial Science, University of Essex}, \orgaddress{\city{Colchester, Essex CO4 3SQ}, \country{U.K.}}
\email{gerald.williams@essex.ac.uk}}



\abstract{We consider a family of cyclic presentations and show that, subject to certain conditions on the defining parameters, they are spines of closed 3-manifolds. These are new examples where the reduced Whitehead graphs are of the same type as those of the Fractional Fibonacci presentations; here the corresponding manifolds are often (but not always) hyperbolic. We also express a lens space construction in terms of a class of positive cyclic presentations that are spines of closed 3-manifolds. These presentations then furnish examples where the Whitehead graphs are of the same type as those of the positive cyclic presentations of type $\mathfrak{Z}$, as considered by McDermott.}

\end{Frontmatter}

\section{Introduction}\label{sec:intro}

The \em cyclically presented group \em $G_n(w)$ is the group defined by the \em cyclic presentation \em
\[ \mathcal{G}_n(w) = \pres{x_0,\ldots, x_{n-1}}{w,\theta(w),\ldots, \theta^{n-1}(w)}\]
where $w(x_0,\ldots ,x_{n-1})$ is a word in the free group $F_n$ with generators $x_0,\ldots ,x_{n-1}$ and $\theta : F_n\rightarrow F_n$ is the \em shift automorphism \em of $F_n$ given by $\theta(x_i)=x_{i+1}$ for each $0\leq i<n$ (subscripts $\bmod~n$, $n>0$). Cyclic presentations that are spines of 3-manifolds have been widely researched, with notable early studies by Dunwoody \cite{Dunwoody}, Sieradski \cite{Sieradski}, Helling, Kim, Mennicke \cite{HKM}, and Cavicchioli and Spaggiari \cite{CavicchioliSpaggiari}.

A necessary condition for a presentation to be a spine of a closed 3-manifold is that its Whitehead graph is planar (see, for example, \cite[page 33]{AngeloniMetzler}, \cite[Documentation, Section 11]{Berge}). The planar, reduced, Whitehead graphs of cyclic presentations were classified in \cite{HowieWilliamsPlanar}, whose labelling  (I.$j$),(II.$j$),(III.$j$) we now adopt. (Types (I.5) and (II.6) are precisely the graphs that correspond to positive or negative cyclic presentations.) In the overwhelming majority of studies of cyclic presentations as spines of 3-manifolds (for example \cite{
AnkaraliogluAydin,
AGMTorus,
CattabrigaMulazzaniStronglyCyclic,
CattabrigaMulazzani,
CattagbrigaMulazzaniReps11knots,
CattagbrigaMulazzaniVesnin,
CHK,
CRStopprop,
CHKTorus,
CavicchioliRuniSpaggiariConjecture,
Dunwoody,
GrasselliMulazanniSeifert11,
GrasselliMulazzani,
HowieWilliamsMFD,
KimKimVesnin98,
KimKimTorusKnots3mfds,
KimKim2bridge,
KimKimPolynomial,
KimKimDualMirror,
KimKimGeneralized,
McDermott,
Mulazzni11Dunwoody,
RSV03,
Sieradski,
SpaggiariTelloni,
SpaggiariNeuwirth,
Telloni10})
the reduced Whitehead graph is of one of the types (I.1), (I.3) or (I.5), which correspond to the graphs given by Dunwoody in {\cite[Figure 1]{Dunwoody}} and, as such, provide examples of so-called \em Dunwoody manifolds\em. By \cite{CattabrigaMulazzani} and \cite{GrasselliMulazzani} the class of Dunwoody manifolds is exactly the class of strongly-cyclic branched covers of $(1,1)$-knots. Strictly speaking, the hypothesis $ab\neq 0$ of \cite{Dunwoody}, which is removed in many later references (including \cite{CattabrigaMulazzani,GrasselliMulazzani}), excludes Whitehead graphs of type (I.5) and for this reason no positive presentations appear in \cite[Table 1]{Dunwoody}. However, an explicit family of cyclic presentations that are spines of 3-manifolds and where the reduced Whitehead graph is of this type are given by the positive presentations of type $\mathfrak{Z}$ considered in \cite[Theorem 8]{McDermott} (that is, those with $s<0$). Cyclic presentations that are spines of 3-manifolds where the reduced Whitehead graph is of a different type to (I.1), (I.3), (I.5) are few and far between. We are only aware of the following families of such presentations. The non-cyclically reduced cyclic presentation $\mathcal{G}_n(x_0x_1x_1^{-1})$ is a spine of $S^3$ for all $n\geq 1$ (\cite[Lemma 9]{HowieWilliamsMFD}). Its Whitehead graph contains loops, but once these are removed the graph is of type (I.4). Helling, Kim and Mennicke \cite{HKM} and Cavicchioli and Spaggiari \cite[Theorem 3]{CavicchioliSpaggiari} showed that, for even $n$, the Fibonacci presentations $\mathcal{F}(n)=\mathcal{G}_n(x_0x_1x_2^{-1})$ are spines of closed, oriented 3-manifolds; here (for $n\geq 4$) the Whitehead graph is of type (II.11). A special case of \cite[Theorem 6]{RuiniSpaggiari98} gives that, for coprime integers $k,l\geq 1$ and even $n$, the cyclic presentations $\mathcal{G}_n(x_0^lx_1^kx_2^{-l})$ are spines of closed, oriented 3-manifolds. These presentations are the \em Fractional Fibonacci presentations \em $\mathcal{F}^{k/l}(n)$ of \cite{KimVesninPreprint,KimVesninFF} (or of \cite{Maclachlan,MaclachlanReid} in the case $l=1$) and if $(k,l)\neq (1,1)$ their reduced Whitehead graph is of type (II.7). Jeong and Wang \cite{Jeong,JeongWang} showed that, for $l\geq 2$ and even $n$ the cyclic presentations $\mathcal{G}_n(x_1(x_2^{-1}x_0)^l)$ are spines of closed, oriented 3-manifolds; here the reduced Whitehead graph is of type (II.14).  

In Section \ref{sec:positivepresentations} we present a family of cyclic presentations that are  spines of closed 3-manifolds where the Whitehead graphs are of type (I.5) and show that the manifolds are cyclic branched covers of a lens space. This construction is essentially well known (see, for example \cite[page 217]{SeifertThrelfall}, \cite[page 4]{Minkus}, or \cite[Section 3]{Kozlovskaya}), but the connection to the planar Whitehead graph classification of \cite{HowieWilliamsPlanar} has not previously been observed. In Section \ref{sec:Gnkl} we present a new family of cyclic presentations that are spines of closed 3-manifolds where the reduced Whitehead graphs are of type (II.7) (i.e. the same type as those of the Fractional Fibonacci presentations) and show that many, but not all, of the manifolds are hyperbolic. Experiments in {\tt Heegaard} \cite{Berge} were used in formulating these results.

\section{Presentation complexes as spines of closed 3-manifolds}\label{sec:3mfdspines}

The \em presentation complex \em (or \em cellular model\em ) $K=K_\mathcal{P}$ of a group presentation $\mathcal{P} = \pres{X}{R}$ is the 2-complex with one 0-cell $O$, a loop at $O$ for each generator $x \in X$ and a 2-cell for each relator (the boundary of that 2-cell spelling the relator). If $N$ is a regular neighbourhood of $O$ then $K \cap \partial N$ is a 1-dimensional cell complex called the \em Whitehead graph \em or \em link graph \em of $\mathcal{P}$. Thus the Whitehead graph of $\mathcal{P}$ is the graph with $2|X|$ vertices $v_x$, $v_x'$ ($x \in  X$) and an edge $(v_x, v_y)$ (resp. $(v_x', v_y')$, $(v_x, v_y')$) for each occurrence of a cyclic subword $xy^{-1}$ (resp. $x^{-1}y$, $(xy)^{\pm 1}$) in a relator $r \in R$. The \em reduced Whitehead graph \em is the graph obtained from the Whitehead graph by replacing all multiedges between two vertices by a single edge. We say that a group presentation $\mathcal{P}$ is the \em spine \em of a closed 3-manifold $M$ if there exists a 3-ball $B^3\subset M$ such that $M-\mathring{B}^3$ collapses onto $K_\mathcal{P}$ (where $\mathring{B}^3$ denotes the interior of $B^3$). Since $K_\mathcal{P}$ is connected, the manifold $M$ is necessarily connected.

Suppose that a group presentation $\mathcal{P}=\pres{X}{R}$ with an equal number of generators and relators is a spine of a closed, oriented 3-manifold. Then (see \cite[Chapter 9]{SeifertThrelfall}, \cite{Neuwirth}, \cite[page 125]{Sieradski}) there is a $3$-complex $C$ whose set of faces consists of precisely one pair of oppositely oriented faces $F_r^+,F_r^-$ for each relator $r\in R$, whose boundaries spell $r$. Let $M_0$ denote the 3-complex obtained from $C$ by identifying the faces $F_r^+,F_r^-$ ($r\in R$). Then $M_0$ is a closed, connected, oriented pseudo-manifold. The \em Seifert-Threlfall condition \em for $M_0$ to be a manifold $M$ is that its Euler characteristic is zero \cite[Theorem I, Section 60]{SeifertThrelfall}. In this case the cell structure on $M=M_0$ has one vertex, one 3-cell, and 2-skeleton homeomorphic to the presentation complex $K_\mathcal{P}$ of $\mathcal{P}$. The Whitehead graph $\Gamma$ of $\mathcal{P}$ is the link of the single vertex of $K_\mathcal{P}$, and so embeds in the link of the single vertex of $M$. The link of the vertex of $M$ is the 2-sphere, and so $\Gamma$ has a planar embedding on this sphere. When $\Gamma$ is connected the 3-complex $C$ is a polyhedron $\pi$ bounding a 3-ball. Given a planar embedding of $\Gamma$ there is a one-to-one correspondence between the faces $F$ of this embedding and the vertices $u_F$ of $\pi$. This correspondence maps a face of degree $d$ to a vertex of degree $d$ of $\pi$. Moreover, if the vertices in the face read cyclically around the face are $w_1,\ldots ,w_d$ where $w_i\in \{v_x,v_x'\ (x\in X)\}$ then the arcs $e_1,\ldots ,e_d$ incident to $u_F$, read cyclically, are directed towards $u_F$ if $w_i=v_x$ and away from $u_F$ if $w_i=v_x'$, and the arc $e_j$, corresponding to vertex $w_j=v_x$ or $v_x'$, is labelled $x$.

\section{Cyclic presentations $\mathcal{H}(r,n)$ as spines of type (I.5)}\label{sec:positivepresentations}

For $n>1,r \geq 1$ let
\[\mathcal{H}(r,n)=\mathcal{G}_n(x_0x_1\ldots x_{r-1})\]
and let $H(r,n)$ be the group it defines. By \cite[Theorems 2 and 3]{Umar}, $H(r,n)$ is finite if and only if $(n,r)=1$, in which case $H(r,n)\cong \Z_r$. The Whitehead graph of $\mathcal{H}(r,n)$ is connected if and only if $r>1$ and $(n,r)=1$, and in this case it is of type (I.5), as shown in Figure \ref{fig:WHGHrn}, where (as also in Figure \ref{fig:WHGGkln}) the edge labels denote their multiplicities, a vertex label $i$ denotes $v_{x_i}$, and a vertex label $\overline{i}$ denotes $v_{x_i}'$. 

As a warm up to our main result, Theorem \ref{thm:Gkl(n,f)spine}, in this section we show that, when the Whitehead graph is connected (that is, if $n, r>1$ and $(n,r)=1$), then $\mathcal{H}(r,n)$ is a spine of a closed, oriented 3-manifold $M$. The positive presentations of type $\mathfrak{Z}$, considered in \cite{McDermott}, also have Whitehead graphs of type (I.5) and, subject to certain hypotheses, they are spines of closed, oriented manifolds $N$ \cite[Theorem 8]{McDermott}. The manifolds $M$ can often be distinguished from the manifolds $N$ by comparing their fundamental groups $\pi_1(M)=H(r,n)\cong \Z_r$ and $\pi_1(N)$. For example, by \cite[Proposition 6]{McDermott}, $\pi_1(N)$ is often infinite and \cite[Table 1]{McDermott} provides computational evidence that many of the finite groups $\pi_1(N)$ are non-cyclic.

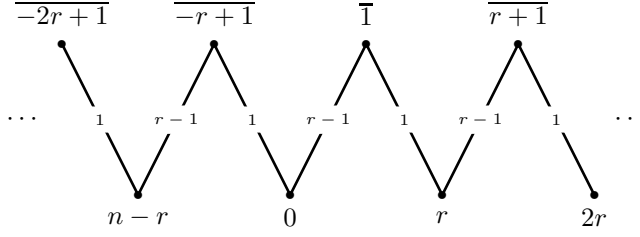
\begin{figure}
\psset{arrowscale=0.7,linewidth=1.0pt,dotsize=3pt}
\begin{center}
\begin{pspicture*}(-2.5,-2.5)(10.5,0.8)
\psline(1.,-2.)(2.,0.)
\psline(3.,-2.)(4.,0.)
\psline(5.,-2.)(6.,0.)

\psline(0.,0.)(1.,-2.)
\psline(2.,0.)(3.,-2.)
\psline(4.,0.)(5.,-2.)
\psline(6.,0.)(7.,-2.)
\begin{scriptsize}
\psdots[dotstyle=*](0.,0.)
\psdots[dotstyle=*](2.,0.)
\psdots[dotstyle=*](4.,0.)
\psdots[dotstyle=*](6.,0.)
\psdots[dotstyle=*](1.,-2.)
\psdots[dotstyle=*](3.,-2.)
\psdots[dotstyle=*](5.,-2.)
\psdots[dotstyle=*](7.,-2.)
\end{scriptsize}
\begin{small}
\rput(1.0,-2.3000000){$n-r$}
\rput(3.0,-2.3000000){$0$}
\rput(5.0,-2.3000000){$r$}
\rput(7.0,-2.3000000){$2r$}
\rput(0,0.4){$\overline{-2r+1}$}
\rput(2,0.4){$\overline{-r+1}$}
\rput(4,0.4){$\overline{1}$}
\rput(6,0.4){$\overline{r+1}$}
\rput(-0.5,-1){$\cdots$}
\rput(7.5,-1){$\cdots$}
\end{small}

\begin{tiny}
\rput*(0.5,-1.0){$1$}
\rput*(1.5,-1.0){$r-1$}
\rput*(2.5,-1.0){$1$}
\rput*(3.5,-1.0){$r-1$}
\rput*(4.5,-1.0){$1$}
\rput*(5.5,-1.0){$r-1$}
\rput*(6.5,-1.0){$1$}
\end{tiny}
\end{pspicture*}
\end{center}
  \caption{Whitehead graph for $\mathcal{H}(r,n)$ (where $(r,n)=1$)\label{fig:WHGHrn} }
\end{figure}

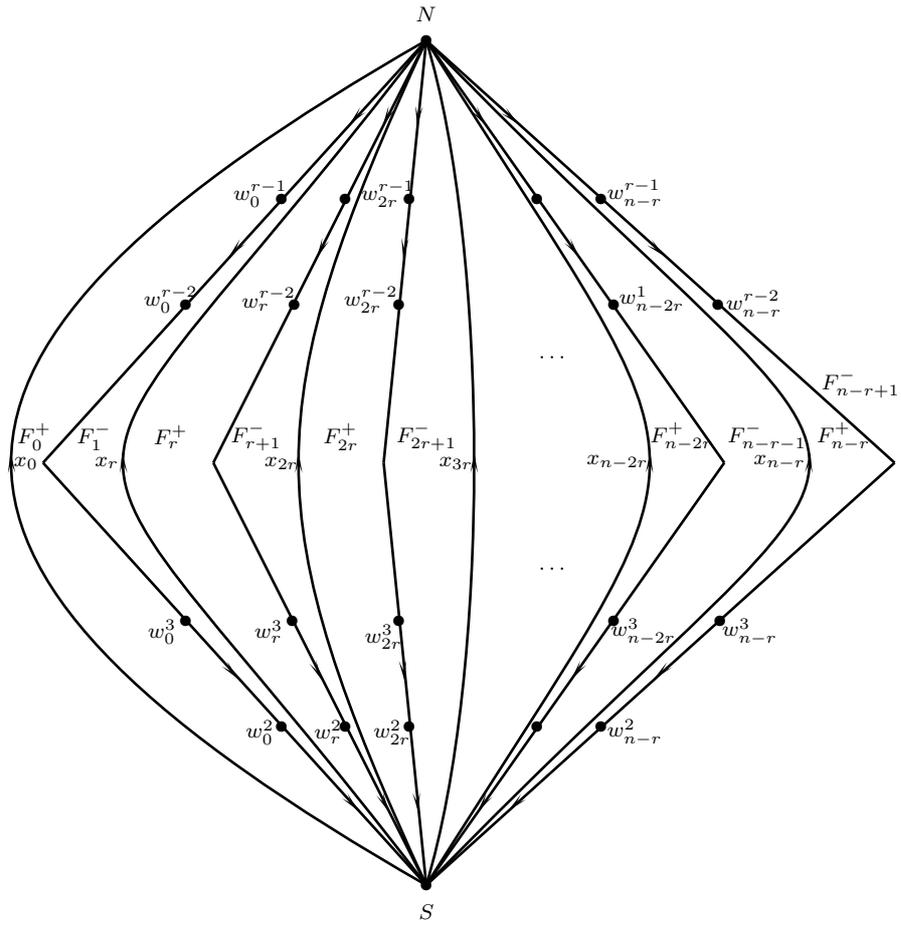
\begin{figure}
\psset{xunit=0.56cm,yunit=0.7cm,algebraic=true,dimen=middle,dotstyle=o,dotsize=5pt 0,linewidth=0.8pt,arrowsize=4pt ,arrowinset=2}
\psset{arrowscale=0.7,linewidth=1.0pt,dotsize=4pt}
    \begin{center}
\begin{pspicture*}(-10,-9)(27,9)
\psline(0.,8.)(-9.,0.)
\psline(0.,8.)(-5.,0.)
\psline(0.,8.)(-1.,0.)
\psline(0.,8.)(7.,0.)
\psline(0.,8.)(11.,0.)
\psline(-9.,0.)(0.,-8.)
\psline(-5.,0.)(0.,-8.)
\psline(-1.,0.)(0.,-8.)
\psline(7.,0.)(0.,-8.)
\psline(11.,0.)(0.,-8.)
\psline[linestyle=none,ArrowInside=->](0.,8.)(-3.4,5.)
\psline[linestyle=none,ArrowInside=->](0.,8.)(-1.9,5.)
\psline[linestyle=none,ArrowInside=->](0.,8.)(-0.4,5.)
\psline[linestyle=none,ArrowInside=->](0.,8.)(2.6,5.)
\psline[linestyle=none,ArrowInside=->](0.,8.)(4.1,5.)
\psline[linestyle=none,ArrowInside=->](-3.4,5.)(-5.65,3.)
\psline[linestyle=none,ArrowInside=->](-1.9,5.)(-3.1,3.)
\psline[linestyle=none,ArrowInside=->](-0.4,5.)(-0.65,3.)
\psline[linestyle=none,ArrowInside=->](2.6,5.)(4.4,3.)
\psline[linestyle=none,ArrowInside=->](4.1,5.)(6.85,3.)
\psline[linestyle=none,ArrowInside=->](-5.65,-3.)(-3.4,-5.)
\psline[linestyle=none,ArrowInside=->](-3.15,-3.)(-1.9,-5.)
\psline[linestyle=none,ArrowInside=->](-0.65,-3.)(-0.4,-5.)
\psline[linestyle=none,ArrowInside=->](4.4,-3.)(2.6,-5.)
\psline[linestyle=none,ArrowInside=->](6.85,-3.)(4.1,-5.)
\psline[linestyle=none,ArrowInside=->](-3.4,-5.)(0,-8.)
\psline[linestyle=none,ArrowInside=->](-1.9,-5.)(0,-8.)
\psline[linestyle=none,ArrowInside=->](-0.4,-5.)(0,-8.)
\psline[linestyle=none,ArrowInside=->](2.6,-5.)(0,-8.)
\psline[linestyle=none,ArrowInside=->](4.1,-5.)(0,-8.)

\psbezier[ArrowInside=->](0,-8)(-13.,-2)(-13.,2)(0,8)
\psbezier[ArrowInside=->](0,-8)(-9.5,2)(-9.5,-2)(0,8)
\psbezier[ArrowInside=->](0,-8)(-4,-1)(-4,1)(0,8)
\psbezier[ArrowInside=->](0,-8)(1.5,-4)(1.5,4)(0,8)
\psbezier[ArrowInside=->](0,-8)(7,1)(7,-1)(0,8)
\psbezier[ArrowInside=->](0,-8)(12,2)(12,-2)(0,8)
\begin{scriptsize}
\rput(-9.2,0.5){$F_{0}^+$}
\rput(-7.8,0.5){$F_{1}^-$}
\rput(-6,0.5){$F_{r}^+$}
\rput(-4,0.5){$F_{r+1}^-$}
\rput(-2,0.5){$F_{2r}^+$}
\rput(+0,0.5){$F_{2r+1}^-$}
\rput(+6,0.5){$F_{n-2r}^+$}
\rput(+8,0.5){$F_{n-r-1}^-$}
\rput(+9.8,0.5){$F_{n-r}^+$}
\rput(+10.2,1.5){$F_{n-r+1}^-$}

\psdots[dotstyle=*](0.,8.) \rput(0,8.5){$N$}
\psdots[dotstyle=*](0.,-8.) \rput(0,-8.5){$S$}
\psdots[dotstyle=*](-3.4,5.) \rput(-3.9,5.1){$w_0^{r-1}$}
\psdots[dotstyle=*](-1.9,5.) 
\psdots[dotstyle=*](-0.4,5.) \rput(-0.9,5.1){$w_{2r}^{r-1}$}
\psdots[dotstyle=*](2.6,5.) 
\psdots[dotstyle=*](4.1,5.) \rput(4.9,5.1){$w_{n-r}^{r-1}$}
\psdots[dotstyle=*](-5.65,3.) \rput(-6.0,3.1){$w_0^{r-2}$}
\psdots[dotstyle=*](-3.1,3.)  \rput(-3.7,3.1){$w_{r}^{r-2}$}
\psdots[dotstyle=*](-0.65,3.)  \rput(-1.3,3.1){$w_{2r}^{r-2}$}
\psdots[dotstyle=*](4.4,3.) \rput(5.3,3.1){$w_{n-2r}^{r-2}$}
\psdots[dotstyle=*](6.85,3.) \rput(7.7,3){$w_{n-r}^{r-2}$}
\psdots[dotstyle=*](-3.4,-5.) \rput(-3.9,-5.1){$w_0^2$}
\psdots[dotstyle=*](-1.9,-5.) \rput(-2.3,-5.1){$w_r^2$}
\psdots[dotstyle=*](-0.4,-5.) \rput(-0.8,-5.1){$w_{2r}^2$}
\psdots[dotstyle=*](2.6,-5.) 
\psdots[dotstyle=*](4.1,-5.) \rput(4.9,-5.1){$w_{n-r}^2$}
%
\psdots[dotstyle=*](-5.65,-3.) \rput(-6.2,-3.2){$w_0^3$}
\psdots[dotstyle=*](-3.15,-3.) \rput(-3.7,-3.2){$w_r^3$}
\psdots[dotstyle=*](-0.65,-3.) \rput(-1.0,-3.3){$w_{2r}^3$}
\psdots[dotstyle=*](4.4,-3.) \rput(5.1,-3.2){$w_{n-2r}^3$}
\psdots[dotstyle=*](6.9,-3.) \rput(7.6,-3.2){$w_{n-r}^3$}
\rput(3,2){$\cdots$}
\rput(3,-2){$\cdots$}
%
\rput(-9.4,0){$x_0$}
\rput(-7.5,0){$x_r$}
\rput(-3.4,0){$x_{2r}$}
\rput(0.7,0){$x_{3r}$}
\rput(4.5,0){$x_{n-2r}$}
\rput(8.3,0){$x_{n-r}$}

\end{scriptsize}
\end{pspicture*}
\end{center}
  \caption{Face pairing polyhedron for $\mathcal{H}(r,n)$\label{fig:P(n)forH(r,n)}}
\end{figure}

\begin{figure}
\psset{arrowscale=0.7,linewidth=1.0pt,dotsize=3pt}
\begin{center}
\begin{pspicture*}(-1.5,-5.0)(14.5,4.5)
\psset{xunit=0.8cm,yunit=1.0cm}
\begin{tiny}
\rput*(1.2,-3.1){$r$}
\rput*(2.05,-3.2){$2$}
\rput*(2.3,-3.4){$1$}
\rput*(4.2,-3.1){$r$}
\rput*(5.05,-3.2){$2$}
\rput*(5.3,-3.4){$1$}
\rput*(7.2,-3.1){$r$}
\rput*(8.05,-3.2){$2$}
\rput*(8.3,-3.4){$1$}
\rput*(0.4,-0.9){$1$}
\rput*(3.4,-0.9){$1$}
\rput*(3.1,-1.){$2$}
\rput*(2.85,-0.9){$3$}
\rput*(6.4,-0.9){$1$}
\rput*(6.1,-1.){$2$}
\rput*(5.85,-0.9){$3$}
\rput*(9.4,-0.9){$1$}
\rput*(9.1,-1.){$2$}
\rput*(8.85,-0.9){$3$}
\rput(2.1,-2){$\cdots$}
\rput(5.1,-2){$\cdots$}
\rput(8.1,-2){$\cdots$}
\end{tiny}

\psline(1.,-4.)(0.,0.)
\psline(1.2,-4.)(2.6,0.)
\psline(1.65,-4.)(3.05,0.)
\psline(1.9,-4.)(3.3,0.)

%
\psline(4.,-4.)(3.0,0.)
\psline(4.2,-4.)(5.6,0.)
\psline(4.65,-4.)(6.05,0.)
\psline(4.9,-4.)(6.3,0.)

\psline(7.,-4.)(6.0,0.)
\psline(7.2,-4.)(8.6,0.)
\psline(7.65,-4.)(9.05,0.)
\psline(7.9,-4.)(9.3,0.)

\psline(10.,-4.)(9.0,0.)
\begin{scriptsize}
\pscircle[fillcolor=white, fillstyle=solid](0.,0.){0.8}
\pscircle[fillcolor=white, fillstyle=solid](3.,0.){0.8}
\pscircle[fillcolor=white, fillstyle=solid](6.,0.){0.8}
\pscircle[fillcolor=white, fillstyle=solid](9.,0.){0.8}
\pscircle[fillcolor=white, fillstyle=solid](1.5,-4.){0.8}
\pscircle[fillcolor=white, fillstyle=solid](4.5,-4.){0.8}
\pscircle[fillcolor=white, fillstyle=solid](7.5,-4.){0.8}
\pscircle[fillcolor=white, fillstyle=solid](10.5,-4.){0.8}
\end{scriptsize}
\begin{small}
\rput(1.5,-4){$F_{n-r}^+$}
\rput(4.5,-4){$F_{0}^+$}
\rput(7.5,-4){$F_{r}^+$}
\rput(10.5,-4){$F_{2r}^+$}
\rput(0,0.0){$F_{-2r+1}^-$}
\rput(3,0.0){$F_{-r+1}^-$}
\rput(6,0.0){$F_{1}^-$}
\rput(9,0.0){$F_{r+1}^-$}
\rput(-0.5,-2){$\cdots$}
\rput(10.5,-2){$\cdots$}
\end{small}
\end{pspicture*}
\end{center}
  \caption{Heegaard diagram for the manifold $M(r,n)$\label{fig:HeegaardforM(r,n)}}
\end{figure}
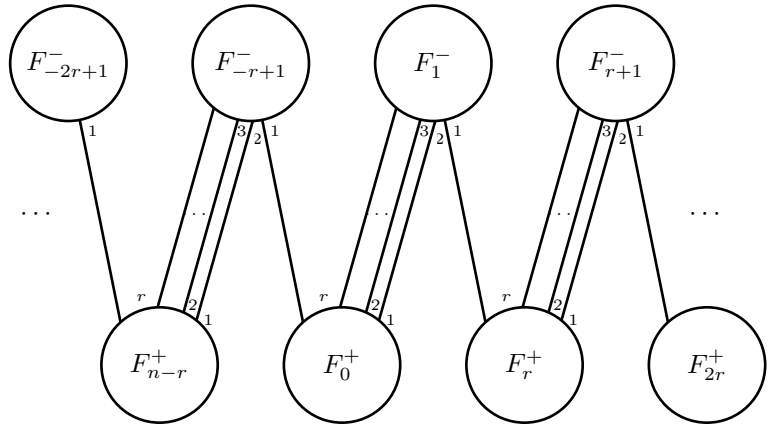

\begin{figure}
\psset{arrowscale=0.7,linewidth=1.0pt,dotsize=3pt}
\begin{center}
\begin{pspicture*}(-3,-0.9)(3,0.9)
\begin{tiny}
\rput*(-1.2,-0.4){$1$}
\rput*(1.2,-0.4){$r$}
\rput*(-1.2,0.4){$3$}
\rput*(1.2,0.4){$2$}
\rput*(-1.2,0.7){$2$}
\rput*(1.2,0.7){$1$}
\rput*(0,0.1){$\vdots$}
\end{tiny}
\psline(-2.,0.5)(2.,0.5)
\psline(-2.,0.3)(2.,0.3)
\psline(-2.,-0.3)(2.,-0.3)
\psline(-2.,-0.5)(2.,-0.5)
\begin{scriptsize}
\pscircle[fillcolor=white, fillstyle=solid](-2.,0.){0.8}
\pscircle[fillcolor=white, fillstyle=solid](2.,0.){0.8}
\end{scriptsize}
\begin{small}
\rput(2.,0.0){$F^+$}
\rput(-2.,0.0){$F^-$}
\end{small}
\end{pspicture*}
\end{center}
  \caption{Heegaard diagram for the manifold $M(r,n)/\rho$\label{fig:HeegaardQuotientforM(r,n)}}
\end{figure}
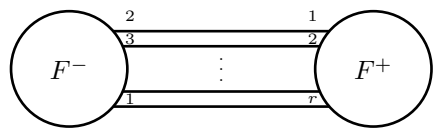

Consider the polyhedron in Figure \ref{fig:P(n)forH(r,n)} with face pairing given by identifying the faces $F_i^+,F_i^-$ ($0\leq i<n)$. We shall show that the identification of faces results in a 3-complex $M$ whose 2-skeleton is the presentation complex of $\mathcal{H}(r,n)$, so $M$ has one 0-cell, $n$ 1-cells, $n$ 2-cells, and one 3-cell, so is a manifold by \cite[Theorem I, Section 60]{SeifertThrelfall}.

Let $[S,N]_i$ denote the arc of the face $F_i^+$ with initial vertex $S$ and terminal vertex $N$. All the arcs labelled $x_0$ are contained in the following cycle (of length $r$):
\begin{alignat*}{1}
 [S,N]_0
&\stackrel{F_{n-(r-1)}}{\longrightarrow} [w_{n-(r-1)}^2,w_{n-(r-1)}^1] 
\stackrel{F_{n-(r-2)}}{\longrightarrow} [w_{n-(r-2)}^3,w_{n-(r-2)}^2]\\ 
&
\stackrel{F_{n-(r-3)}}{\longrightarrow} [w_{n-(r-3)}^4,w_{n-(r-3)}^3] 
\stackrel{F_{n-(r-4)}}{\longrightarrow} \cdots 
\stackrel{F_{n-2}}{\longrightarrow} [w_{n-2}^{r-1},w_{n-2}^{r-2}]\\ 
&
\stackrel{F_{n-1}}{\longrightarrow} [N,w_{n-1}^{r-1}]
\stackrel{F_{0}}{\longrightarrow} [S,N]_0. 
\end{alignat*}

All the vertices that are a vertex (either initial or terminal) of an arc labelled $x_0$ are contained in the (induced) cycle of the initial vertices in the arcs above, so in the resulting complex $M$ these vertices are identified. In particular, vertices $N,S$ are identified, and (by comparing initial vertices) for $1< j < r$ vertices $w_{n-j}^{r-j+1}$ are identified with $S$. Therefore, for each $0\leq i<n$, $1< j<r$ the vertex $w_i^j=\theta^{i+r-j+1}(w_{n-(r-j+1)}^j)$ is identified with $S$. Therefore all the vertices of the polyhedron are identified. Thus the quotient $M$ has one 3-cell, $n$ 2-cells, $n$ 1-cells, and one $0$-cell, and since the boundaries of the $2$-cells spell the relators of $\mathcal{H}(r,n)$, it follows that $M$ is a closed, oriented 3-manifold, and $\mathcal{H}(r,n)$ is a spine of $M$.

As in (for example) \cite[proof of Proposition 2]{KozlovskayaVesnin}, a Heegaard diagram for $M$ arising from the face pairing polyhedron is given in Figure \ref{fig:HeegaardforM(r,n)}. This diagram has a rotational symmetry $\rho$ of order $n$ cyclically permuting the faces $F_i^+\rightarrow F_{i+r}^+$ and $F_i^-\rightarrow F_{i+r}^-$. The quotient of $M$ by $\rho$ is a 3-orbifold in which the image of the rotation axis is a singular set. This yields the Heegaard diagram for $M/\rho$ in Figure \ref{fig:HeegaardQuotientforM(r,n)}, which is the canonical diagram of the lens space $L(r,1)$. Hence the manifold $M$ is an $n$-fold  cyclic branched cover of $L(r,1)$.

\section{Cyclic presentations $\mathcal{G}^{k/l}(n,f)$ as spines of type (II.7)}\label{sec:Gnkl}

For $n\geq 2$, $k,l\geq 1$, $0\leq f<n$, let $\mathcal{G}^{k/l}(n,f)$ be the cyclic presentation
\[\mathcal{G}_n((y_0y_f\ldots y_{(l-1)f}) (y_{lf+1}y_{(l+1)f+1}\ldots y_{(l+(k-1))f+1}) (y_2y_{2+f}\ldots y_{2+(l-1)f})^{-1})\]
and let $G^{k/l}(n,f)$ be the group it defines. When $n\geq 4$, the Whitehead graph of $\mathcal{G}^{k/l}(n,f)$ is planar if and only if $n$ is even and either $fk\equiv 0 \bmod n$ or $fk\equiv 2 \bmod n$, in which case it is of type (II.7) if $(k,l)\neq (1,1)$ and is of type (II.11) if $k=l=1$; see Figure \ref{fig:WHGGkln}. Our main result is the following:

\begin{figure}
\psset{arrowscale=0.7,linewidth=1.0pt,dotsize=3pt}
\psset{yunit=0.8cm}
\begin{center}
    \begin{pspicture*}(-2,-2.5)(11,2.5)
\psline(-0.5,-2.)(1.5,-2.)
\psline(1.5,-2.)(3.5,-2.)
\psline(3.5,-2.)(5.5,-2.)
\psline(5.5,-2.)(7.5,-2.)
\psline(0.5,2.)(1.,0.)
\psline(2.5,2.)(3.,0.)
\psline(4.5,2.)(5.,0.)
\psline(6.5,2.)(7.,0.)
\psline(0.5,2.)(0.,0.)
\psline(2.5,2.)(2.,0.)
\psline(4.5,2.)(4.,0.)
\psline(6.5,2.)(6.,0.)
\psline(1.,0.)(1.5,-2.)
\psline(3.,0.)(3.5,-2.)
\psline(5.,0.)(5.5,-2.)
\psline(7.,0.)(7.5,-2.)

\psline(0.,0.)(-0.5,-2.)
\psline(2.,0.)(1.5,-2.)
\psline(4.,0.)(3.5,-2.)
\psline(6.,0.)(5.5,-2.)
\psline(8.,0.)(7.5,-2.)
%
%
\psline(0.,2.)(2.,2.)
\psline(2.,2.)(4.,2.)
\psline(4.,2.)(6.,2.)
\psline(6.,2.)(8.,2.)
\psline(0.,0.)(2.,0.)
\psline(2.,0.)(4.,0.)
\psline(4.,0.)(6.,0.)
\psline(6.,0.)(8.,0.)
\begin{scriptsize}
\psdots[dotstyle=*](0.5,2.)
\psdots[dotstyle=*](2.5,2.)
\psdots[dotstyle=*](4.5,2.)
\psdots[dotstyle=*](6.5,2.)
\psdots[dotstyle=*](0.,0.)
\psdots[dotstyle=*](1.,0.)
\psdots[dotstyle=*](2.,0.)
\psdots[dotstyle=*](3.,0.)
\psdots[dotstyle=*](4.,0.)
\psdots[dotstyle=*](5.,0.)
\psdots[dotstyle=*](6.,0.)
\psdots[dotstyle=*](7.,0.)
\psdots[dotstyle=*](8.,0.)
\psdots[dotstyle=*](-0.5,-2.)
\psdots[dotstyle=*](1.5,-2.)
\psdots[dotstyle=*](3.5,-2.)
\psdots[dotstyle=*](5.5,-2.)
\psdots[dotstyle=*](7.5,-2.)
\end{scriptsize}

\begin{tiny}
\rput*(-0.25,-1.0){$\lambda$}
\rput*(1.75,-1.0){$\lambda$}
\rput*(3.75,-1.0){$\lambda$}
\rput*(5.75,-1.0){$\lambda$}
\rput*(7.75,-1.0){$\lambda$}
\rput*(0.75,1.0){$\lambda$}
\rput*(2.75,1.0){$\lambda$}
\rput*(4.75,1.0){$\lambda$}
\rput*(6.75,1.0){$\lambda$}
\rput*(1.25,-1.0){$1$}
\rput*(3.25,-1.0){$1$}
\rput*(5.25,-1.0){$1$}
\rput*(7.25,-1.0){$1$}
\rput*(0.25,1.0){$1$}
\rput*(2.25,1.0){$1$}
\rput*(4.25,1.0){$1$}
\rput*(6.25,1.0){$1$}
%
\rput*(0.5,-2){$1$}
\rput*(2.5,-2){$1$}
\rput*(4.5,-2){$1$}
\rput*(6.5,-2){$1$}
\rput*(1.5,2){$1$}
\rput*(3.5,2){$1$}
\rput*(5.5,2){$1$}
\rput*(7.5,2){$1$}
\rput*(0.5,0){$1$}
\rput*(1.5,0){$1$}
\rput*(2.5,0){$1$}
\rput*(3.5,0){$1$}
\rput*(4.5,0){$1$}
\rput*(5.5,0){$1$}
\rput*(6.5,0){$1$}
\rput*(7.5,0){$1$}
\end{tiny}

\begin{footnotesize}
%
\rput(-0.5,-2.3000000){$\overline{f-4}$}
\rput(1.5,-2.3000000){$\overline{f-2}$}
\rput(3.5,-2.3000000){$\overline{f}$}
\rput(5.5,-2.3000000){$\overline{f+2}$}
\rput(7.5,-2.3000000){$\overline{f+4}$}
\rput(-0.3,0.2){$-4$}
\rput(0.7,-0.2){$-3$}
\rput(1.7,0.2){$-2$}
\rput(2.7,-0.2){$-1$}
\rput(3.8,0.2){$0$}
\rput(4.9,-0.2){$1$}
\rput(5.9,0.2){$2$}
\rput(6.8,-0.2){$3$}
\rput(7.8,0.2){$4$}
\rput(0.5,2.3000000){$\overline{f-3}$}
\rput(2.5,2.3000000){$\overline{f-1}$}
\rput(4.5,2.3000000){$\overline{f+1}$}
\rput(6.5,2.3000000){$\overline{f+3}$}
\rput(-1,-1){$\cdots$}
\rput(-1,1){$\cdots$}
\rput(9,-1){$\cdots$}
\rput(9,1){$\cdots$}
\end{footnotesize}
\end{pspicture*}
\end{center}
  \caption{Whitehead graph for $\mathcal{G}^{k/l}(n,f)$ (where $\lambda = 2l+k-3$, $n$ even and $fk\equiv 0$ or $2\bmod n$)\label{fig:WHGGkln}}
\end{figure}
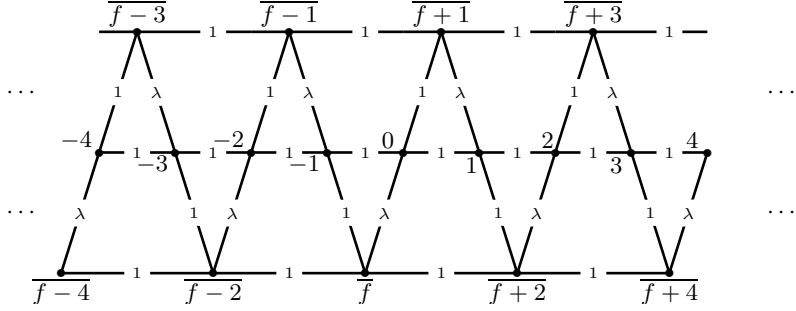

\begin{maintheorem}\label{thm:Gkl(n,f)spine}
Let $n\geq 4$, $fk\not \equiv 2 \bmod n$, and $(k,l)\in \{(k,1),(1,l),(5,2),(2,5)\}$. Then $\mathcal{G}^{k/l}(n,f)$ is a spine of a closed, oriented 3-manifold $M^{k/l}(n,f)$ if and only if $n$ and $f$ are even and $fk\equiv 0 \bmod n$.
\end{maintheorem}

Our motivation for studying these presentations (and, in particular, under the condition $fk\equiv 0 \bmod n$ but not under the condition $fk\equiv 2 \bmod n$) stems from the following connection to the Fractional Fibonacci groups. The presentations $\mathcal{G}^{k/l}(n,0)$ are the Fractional Fibonacci group presentations 
\[\mathcal{F}^{k/l}(n)= \mathcal{G}_n(x_0^lx_1^kx_2^{-l})\]
introduced in \cite{Maclachlan,MaclachlanReid,KimVesninPreprint,KimVesninFF}, generalizing the Fibonacci group presentations $\mathcal{F}^{1/1}(n)$. For even $n$ and coprime integers $k,l\geq 1$, the Fractional Fibonacci group $F^{k/l}(n)$ is a 3-manifold group \cite[Theorem 4.1]{Maclachlan}, \cite[Section 4]{MaclachlanReid},  \cite[Section 2]{KimVesninPreprint}, \cite[Section 2]{KimVesninFF}. If $n\geq 2$ is even then $\mathcal{F}^{1/1}(n)$ was shown to be a spine of a closed, oriented, 3-manifold in \cite{CavicchioliSpaggiari},\cite{HKM},\cite{HLM1},\cite{HLM2}. More generally, if $k,l\geq 1$ are coprime and $n$ is even, then setting $s_i=q_i=l$, $p_i=-r_i=k$ in \cite[Theorem 6]{RuiniSpaggiari98} gives that $\mathcal{F}^{k/l}(n)$ is a spine of a closed, oriented, 3-manifold. If $n$ is odd then $\mathcal{F}^{k/l}(n)$ is not a spine, since its Whitehead graph is non-planar. The question as to when, for odd $n$, $F^{k/l}(n)$ is a 3-manifold group is considered in \cite[Theorem 3]{HowieWilliamsMFD} for the case $k=l=1$, and in \cite[Theorem 6.2]{ChinyereWilliamsFF} for the case $l=1$ and the case $n=3$. The abelianisation $F^{k/1}(n)^\mathrm{ab}$ is obtained in \cite[Lemma 1]{MaclachlanReid}, \cite[Corollary 4.3]{NoferiniWilliams}.

The \em shift extension \em $E^{k/l}(n)=F^{k/l}(n) \rtimes_\theta \Z_n$  of $F^{k/l}(n)$ by $\Z_n=\pres{t}{t^n}$ has the presentation 
\begin{alignat}{1}
E^{k/l}(n)=\pres{x,t}{t^n, x^ltx^ktx^{-l}t^{-2}},\label{eq:extension}
\end{alignat}
where the second relator is obtained by rewriting the relators of $F^{k/l}(n)$ in terms of the substitutions $x_i=t^ixt^{-i}$ (see, for example, \cite[Theorem 4]{JWW}). As set out in \cite[Section 2]{BogleyShift}, for $0\leq f<n$ such that $fk\equiv 0 \bmod n$, there is a retraction $\nu^f:E^{k/l}(n) \rightarrow \Z_n=\pres{t}{t^n}$ given by $\nu^f(t)=t$, $\nu^f(x)=t^f$. The kernel, $\mathrm{ker} (\nu^f)$,  has a presentation with generators $y_i=t^ixt^{-(i+f)}$ ($0\leq i<n$) and relators that are rewrites of conjugates of the second relator of $E^{k/l}(n)$ by powers of $t$, and so has the cyclic presentation $\mathcal{G}^{k/l}(n,f)$. In particular (under the condition $fk\equiv 0 \bmod n$), 
\begin{alignat*}{1}
  \mathcal{G}^{k/1}(n,0) &= \mathcal{G}_n(y_0 y_1^{k} y_2^{-1})=\mathcal{F}^{k/1}(n),\\
  \mathcal{G}^{1/l}(n,f) &=\mathcal{G}^{1/l}(n,0)=\mathcal{G}_n(y_0^l y_1 y_2^{-l})=\mathcal{F}^{1/l}(n),\\
  \mathcal{G}^{5/2}(n,f) &= \mathcal{G}_n( (y_0y_f) (y_{2f+1}y_{3f+1}y_{4f+1}y_1y_{f+1}) (y_2y_{2+f})^{-1} ),\\
  \mathcal{G}^{2/5}(n,f) &= \mathcal{G}_n((y_0y_fy_{0}y_{f}y_{0}) (y_{f+1}y_1) ( y_2y_{2+f}y_2y_{2+f}y_2)^{-1}).
\end{alignat*}
The shift extension
\begin{alignat*}{1}
     G^{k/l}(n,f) \rtimes_\theta \Z_n 
     &= \pres{y,t}{t^n, (yt^f)^lt(yt^f)^kt^{-fk+1}(yt^f)^{-l}t^{-2}}\\
     &= \pres{x,t}{t^n, x^ltx^kt^{-fk+1}x^{-l}t^{-2}}
\end{alignat*}
(by setting $x=yt^f$ and eliminating $y$). In the case $fk\equiv 0 \bmod n$, this coincides with $E^{k/l}(n)$, which is independent of the value of $f$. This implies, for example, that the order of $G^{k/l}(n,f)$ is independent of $f$, the shift dynamics of $G^{k/l}(n,f)$ are identical for all values of $f$ \cite[Lemma 2.2]{BogleyShift}, and that $G^{k/l}(n,f)$ is (non-elementary) word hyperbolic if and only if $F^{k/l}(n)$ is (non-elementary) word hyperbolic. In Theorem \ref{thm:commensurablehyperbolic} we similarly show that, for fixed $k,l,n$, if two groups $G^{k/l}(n,f_1),G^{k/l}(n,f_2)$ sharing a shift extension are fundamental groups of closed, connected, orientable 3-manifolds then either both manifolds are hyperbolic, or neither are. 
Observe that (under the hypothesis $fk\equiv 0\bmod n$), as in Figure \ref{fig:WHGGkln}, the Whitehead graph of $\mathcal{G}^{k/l}(n,f)$ has vertices $v_{y_i},v_{y_i}'$ and edges $(v_{y_i},v_{y_{i+1}})$, $(v_{y_i}',v_{y_{i+2}}')$, $(v_{y_i},v_{y_{i+f}}')$ (of multiplicity $2l+k-3$), and so the Whitehead graph of $\mathcal{G}^{k/l}(n,f)$ is obtained from that of $\mathcal{F}^{k/l}(n)=\mathcal{G}^{k/l}(n,0)$ by replacing each edge $(v_{y_i},v_{y_{j}}')$ by the edge $(v_{y_i},v_{y_{j+f}}')$.

\begin{remark}\label{rem:WHGcommensurable}
In the general setting of \cite{BogleyShift}, if $\nu^f:\pres{x,t}{t^n,W(x,t)} \rightarrow \pres{t}{t^n}$ ($n\geq 2$) is a retraction given by $\nu^{f}(t)=t$ and $\nu^{f}(x)=t^{f}$ then $\mathrm{ker}(\nu^f)$ has a cyclic presentation $\mathcal{G}_n(\rho^f(W(x,t)))$ where $\rho^f(W(x,t))$ is as defined in \cite[page 159]{BogleyShift}. Analysis of the length two subwords $x_{u(\iota)}^{\epsilon_\iota}x_{u(\iota+1)}^{\epsilon_{\iota+1}}$ of $\rho^f(W(x,t))$ (where $\epsilon_\iota=\pm 1, \epsilon_{\iota+1}=\pm 1$) provides a description of the edge set of the Whitehead graph of $\mathcal{G}_n(\rho^f(W(x,t)))$. Given two such retractions $\nu^{f_1},\nu^{f_2}$ this description yields that the Whitehead graph of $\mathcal{P}_2=\mathcal{G}_n(\rho^{f_2}(W(x,t)))$ is obtained from that of $\mathcal{P}_1=\mathcal{G}_n(\rho^{f_1}(W(x,t)))$ by replacing each edge $(v_{x_i},v_{x_j}')$ by the edge $(v_{x_i},v_{x_{j+(f_2-f_1)}}')$, leaving all other edges unchanged. In particular, if the Whitehead graph of $\mathcal{P}_1$ is one of the types in the planarity classification of \cite{HowieWilliamsPlanar} then the Whitehead graph of $\mathcal{P}_2$ is of the same type.
\end{remark}

In contrast, for fixed $k,l,n$, two groups $G^{k/l}(n,f_1), G^{k/l}(n,f_2)$ sharing a shift extension will typically be non-isomorphic. We give examples of this in Example \ref{ex:finitegroups} and Lemma \ref{lem:Fkl(n)vsGkl(n,n/2)}. Further, in the case $fk\equiv 0 \bmod n$, we also have that the presentation $\mathcal{G}^{k/l}(n,f)$ has a planar Whitehead graph. Since the group $G^{k/l}(n,f)$ shares its shift extension with that of $F^{k/l}(n)$, these properties combined suggest that the geometric properties of $F^{k/l}(n)$ and $\mathcal{F}^{k/l}(n)$ may be inherited by $G^{k/l}(n,f)$ and $\mathcal{G}^{k/l}(n,f)$. However, while \cite[Theorem 6]{RuiniSpaggiari98} implies that $\mathcal{F}^{k/l}(n)$ is a spine for all coprime $k,l$ and even $n\geq 2$, Lemma \ref{lem:notspine} will show that, additionally, $f$ must be even for $\mathcal{G}^{k/l}(n,f)$ to be a spine of a closed, oriented, 3-manifold. This demonstrates that, generally, given two retractions $\nu^{f_1},\nu^{f_2}:\pres{x,t}{t^n,W(x,t)}\rightarrow \pres{t}{t^n}$ whose kernels have cyclic presentations $\mathcal{P}_1,\mathcal{P}_2$, as in Remark \ref{rem:WHGcommensurable}, it is not the case that $\mathcal{P}_1$ is a spine of a closed, oriented 3-manifold if and only if $\mathcal{P}_2$ is such a spine. In the other planar case, $fk\equiv 2\bmod n$, the groups ${G}^{k/l}(n,f)$ do not share a shift extension with ${F}^{k/l}(n)$, so there is no apriori reason to expect similar geometric properties to hold. 

\begin{example}[Non-cyclic finite groups.]\label{ex:finitegroups}
The groups $G^{3/1}(3,0),G^{3/1}(3,1),$ are solvable of order 3528, and derived lengths 3, and 4, respectively (\cite[Remark 1]{MaclachlanReid},\cite[page 282]{JohnsonRobertson}) and hence are non-isomorphic. The groups $G^{3/2}(3,0),G^{3/2}(3,1),$ are solvable of order 504, and derived lengths 2, and 3 respectively (using \cite{GAP}) and hence are non-isomorphic. 
\end{example}

In Lemma \ref{lem:Fkl(n)vsGkl(n,n/2)} we use the fact that the order of the abelianisation $G^{k/l}(n,f)^\mathrm{ab}$ is given by $|\mathrm{Res}(p_f(t),t^n-1)|$, if this is non-zero, and is infinite otherwise, where
\[ p_f(t)= (1-t^2)(1+t^f+\ldots +t^{(l-1)f}) + t^{lf+1}(1+t^f+\ldots + t^{(k-1)f})\]
is the \em representer polynomial \em of $G^{k/l}(n,f)$, and $\mathrm{Res}(\cdot, \cdot)$ denotes the resultant \cite[page 82]{Johnson}. As we will only be interested in the absolute values of resultants (and not the sign), to avoid repetitive use of modulus signs we will take $\mathrm{Res}(\cdot , \cdot)$ to mean $|\mathrm{Res}(\cdot , \cdot)|$.

\begin{lemma}\label{lem:Fkl(n)vsGkl(n,n/2)}
Let $n,k,l\geq 1$ where $n,k$ are even and $\gcd(k,l)=1$  then $F^{k/l}(n)=G^{k/l}(n,0)\not \cong G^{k/l}(n,n/2)$.
\end{lemma}

\begin{proof}
For any $0\leq f<n$
\begin{alignat}{1}
    \mathrm{Res}(p_{f}(t),t^n-1)
= \mathrm{Res}(p_{f}(t),t^{n/2}-1)\cdot \mathrm{Res}(p_f(t),t^{n/2}+1).\label{eq:resultant}
\end{alignat}
The hypotheses imply that $l$ is odd so we have
\begin{alignat*}{1}
\mathrm{Res}(p_{n/2}(t),t^{n/2}-1)
&= \mathrm{Res}(l(1-t^2) + kt,t^{n/2}-1)\\
&= \mathrm{Res}(p_0(t),t^{n/2}-1),
\end{alignat*}
and, setting $\alpha,\bar{\alpha}=(k\pm \sqrt{k^2+4l^2})/(2l)$,
\begin{alignat*}{1}
\mathrm{Res}(p_{0}(t),t^{n/2}+1)
&=\mathrm{Res}\left( l(t-\alpha)(t-\bar{\alpha}), t^{n/2}+1\right)\\
&= l^{n/2} (\alpha^{n/2}+1)(\bar{\alpha}^{n/2}+1)\\
&= l^{n/2}\left((-1)^{n/2}+1+(\alpha^{n/2} + \bar{\alpha}^{n/2})\right).
\end{alignat*}
On the other hand,
\begin{alignat*}{1}
\mathrm{Res}(p_{n/2}(t),t^{n/2}+1)
&=\mathrm{Res}\left( l (1-t^2) + kt(1+t^{n/2})/2, t^{n/2}+1\right)\\
&=l^{n/2}\cdot 2(1+(-1)^{n/2}).
\end{alignat*}
Therefore $\mathrm{Res}(p_0(t),t^{n/2}+1)\neq \mathrm{Res}(p_{n/2}(t),t^{n/2}+1)$ so, by equation (\ref{eq:resultant}), $\mathrm{Res}(p_0(t),t^{n}-1)\neq \mathrm{Res}(p_{n/2}(t),t^{n}-1)$. Hence $|G^{k/l}(n,0)^\mathrm{ab}| \neq |G^{k/l}(n,n/2)^\mathrm{ab}|$, and the result follows.
\end{proof}

\begin{lemma}\label{lem:notspine}
Suppose $n\geq 4$, $k,l\geq 1$, $0\leq f<n$ where $n$ is even, $fk\equiv 0\bmod n$, and $\gcd(k,l)=1$. If $f$ is odd then $\mathcal{G}^{k/l}(n,f)$ is not a spine of a closed, oriented 3-manifold.
\end{lemma}

\begin{proof}
First note that if $f$ is odd, then the hypotheses imply that $l$ is odd. The Whitehead graph of $\mathcal{G}^{k/l}(n,f)$ has a planar embedding, as shown in Figure \ref{fig:WHGGkln}. It has two $n/2$-gons $v_{y_0}'-v_{y_2}'-\cdots -v_{y_{n-2}}'-v_{y_0}'$ and $v_{y_1}'-v_{y_3}'-\cdots -v_{y_{n-1}}'-v_{y_1}'$, $n$ 3-gons $v_{y_i}-v_{y_{i+1}}-v_{y_{i+f+1}}'-v_{y_i}$, $n$ 4-gons $v_{y_i}-v_{y_{i+1}}-v_{y_{i+f+2}}'-v_{y_{i+f}}'-v_{y_i}$, and $\lambda n$ 2-gons $v_{y_i}-v_{y_{i+f+1}}'-v_{y_i}$, where $\lambda= 2l+k-3$. It follows that the reduced Whitehead graph is unique up to self homeomorphism of $S^2$.

If $\mathcal{G}^{k/l}(n,f)$ is a spine of a closed, oriented, 3-manifold then in a putative corresponding face-pairing polyhedron, there are two degree $n/2$ source vertices, $N,S$, say,  where $N$  has outgoing  arcs in cyclic order $y_0,y_2,y_4,\ldots , y_{n-2}$ and $S$ has outgoing arcs in cyclic order $y_1,y_3,y_5,\ldots , y_{n-1}$.  The 2-cells incident to $N$ are therefore, in cyclic order, $F_0^-,F_2^-,F_4^-, \dots , F_{n-2}^-$, where the boundary of $F_i^-$ reads, anticlockwise, the $i$'th relator of $\mathcal{G}^{k/l}(n,f)$. The arc labelled $y_2$, with initial vertex $N$, is the first of a path $y_2-y_{2+f}-\cdots -y_{2+(l-1)f}$. Let $v$ denote the terminal vertex of the last arc in this path. Then $v$ has two incoming arcs labelled  $y_{(l-1)f+1},y_{(l-1)f+2}$, and an outgoing arc labelled $y_{lf+3}$. Therefore $v$ is a degree 4 vertex, and its remaining arc is outgoing, and labelled $y_{lf+1}$. The outgoing arcs $y_{lf+1},y_{lf+3}$ then bound the $F_{lf+1}^-$ face. Since $l,f$ are odd, $lf+1$ is even, so the face pairing contains two $F_{lf+1}^-$ faces, a contradiction.
\end{proof}

In Figures \ref{fig:Pk1(n)},\ref{fig:P52(n)},\ref{fig:P1l(n)},\ref{fig:P25(n)} we present face-pairing polyhedra with $n$ pairs of faces $F_i^+,F_i^-$ ($0\leq i<n$) of opposite orientation whose boundaries spell the defining relators of the relevant presentation $\mathcal{G}^{k/l}(n,f)$. In the proof of Theorem \ref{thm:Gkl(n,f)spine} we show that the identification of faces $F_i^+,F_i^-$ results in a 3-complex that satisfies the Seifert-Threlfall condition, and so is a manifold whose spine is the presentation complex of $\mathcal{G}^{k/l}(n,f)$. (Note that, in these figures, the condition that $f$ is even is necessary for each relator to appear as the label of a pair of oppositely oriented faces.) The diagrams have different forms depending on whether $l>k$ or $k>l$. The figures should provide the enthusiastic reader with sufficient information to construct face-pairing polyhedra in the general cases. The values $\{k,l\}=\{2,5\}$ were selected so that $k,l,|k-l|$ are distinct, to make it straightforward to infer paths of lengths $k,l$, and $|k-l|$. (For example, in Figure \ref{fig:P52(n)} the path from $v_{2f}$ to $u_2$ has length $3=k-l$, and in Figure \ref{fig:P25(n)} the path from $w_f$ to $v_{f-1}$ has length $3=l-k$.) We expect the Seifert-Threlfall condition to hold in the general case for coprime $k,l$; however, the corresponding analysis to check this would be highly technical. Since our goal is to exhibit new cyclic presentations that arise as spines of closed 3-manifolds, whose Whitehead graphs are of type (II.7), rather than to be exhaustive, we have not sought to do this.

\begin{altproof}
We may assume $n$ is even and $fk\equiv 0$, for otherwise the Whitehead graph of $\mathcal{G}^{k/l}(n,f)$ is not planar \cite{HowieWilliamsPlanar} (noting that $fk\not \equiv 2 \bmod n$, by hypothesis). Moreover, by Lemma \ref{lem:notspine} we may assume $f$ is even.

\noindent \textit{Case 1: $(k,l)=(k,1)$.}
Consider the face pairing polyhedron depicted in Figure \ref{fig:Pk1(n)} (where the arc labels can be deduced from the position of the unique source in the boundary of each face) with face pairing given by identifying the faces $F_i^+,F_i^-$ ($0\leq i<n)$. We shall show that the identification of faces results in a 3-complex $M$ whose $2$-skeleton is the presentation complex of $\mathcal{G}^{k/l}(n,f)$, so $M$ has one $0$-cell, $n$ 1-cells, $n$ 2-cells, and one $3$-cell, so is a manifold by \cite[Theorem I, Section 60]{SeifertThrelfall}.

All the arcs labelled $x_0$ are contained in the following cycle: 
\begin{alignat*}{1}
[N,u_0]
&\stackrel{F_0}{\longrightarrow} [u_{1-f},v_0]
\stackrel{F_{-1}}{\longrightarrow} [w_{f}^{k-2},u_1]
\stackrel{F_{f-1}}{\longrightarrow} [w_{2f}^{k-3},w_{2f}^{k-2}]\\
&\stackrel{F_{2f-1}}{\longrightarrow} [w_{3f}^{k-4},w_{3f}^{k-3}]
\stackrel{F_{3f-1}}{\longrightarrow} \cdots
\stackrel{F_{(k-3)f-1}}{\longrightarrow} [w_{(k-2)f}^{1},w_{(k-2)f}^{2}]\\
&\stackrel{F_{(k-2)f-1}}{\longrightarrow} [v_{(k-1)f-2},w_{(k-1)f}^{1}]
\stackrel{F_{(k-1)f-1}}{\longrightarrow} [v_{(k-1)f-3},w_{0}^1]
\stackrel{F_{-2}}{\longrightarrow} [N,u_0].
\end{alignat*}
Therefore, in the resulting complex $M$ all the arcs labelled $x_0$ are identified. Moreover, all the vertices that are a vertex (either initial or terminal) of an arc labelled $x_0$ are contained in the (induced) cycle of the initial vertices in the arcs above, so these vertices are identified in $M$. By applying the shift $\theta^{2j}$ ($0\leq j<n/2$) to the above, all the arcs labelled $x_{2j}$ are identified and all the vertices of such arcs are identified. That is, the vertices $N,u_{2j+(1-f)},w_{2j+f}^{k-2},w_{2j+2f}^{k-3},\dots , w_{2j+(k-2)f}^1, v_{2j+(k-1)f-2}$ are identified; equivalently, $N$, $u_{2j+1}, w_{2j}^1,\dots w_{2j}^{k-2}, v_{2j}$ ($0\leq j<n/2$) are identified.

All the arcs labelled $x_1$ are contained in the following cycle:
\begin{alignat*}{1}
[S,u_1]
&\stackrel{F_1}{\longrightarrow} [u_{2-f},v_1]
\stackrel{F_{0}}{\longrightarrow} [w_{f+1}^{k-2},u_2]
\stackrel{F_{f}}{\longrightarrow} [w_{2f+1}^{k-3},w_{2f+1}^{k-2}]\\
&\stackrel{F_{2f}}{\longrightarrow} [w_{3f+1}^{k-4},w_{3f+1}^{k-3}]
\stackrel{F_{3f}}{\longrightarrow} \cdots
\stackrel{F_{(k-3)f}}{\longrightarrow} [w_{(k-2)f+1}^{1},w_{(k-2)f+1}^{2}]\\
&\stackrel{F_{(k-2)f}}{\longrightarrow} [v_{(k-1)f-1},w_{(k-1)f+1}^{1}]
\stackrel{F_{(k-1)f}}{\longrightarrow} [v_{(k-1)f-2},w_{1}^1]
\stackrel{F_{-2}}{\longrightarrow} [S,u_1].
\end{alignat*}
Therefore, in $M$ all the arcs labelled $x_1$ are identified. Moreover, all the vertices that are a vertex (either initial or terminal) of an arc labelled $x_1$ are contained in the (induced) cycle of the initial vertices in the arcs above, so these vertices are identified in $M$. By applying the shift $\theta^{2j}$ ($0\leq j<n/2$) to the above, all the arcs labelled $x_{2j+1}$ are identified and all the vertices of such arcs are identified. That is, the vertices $S,u_{2j+(2-f)},w_{2j+f+1}^{k-2},w_{2j+2f+1}^{k-3},\dots , w_{2j+(k-2)f+1}^1, v_{2j+(k-1)f-1}$ are identified; that is $S$, $u_{2j}, w_{2j+1}^1,\dots w_{2j+1}^{k-2}, v_{2j+1}$ ($0\leq j<n/2$) are identified. Moreover, examining the terminal vertices of the first cycle above we see that $u_0$ and $u_1$ are identified. Therefore the vertices induced from the first cycle are identified with the vertices from the second cycle, so all vertices of the polyhedron are identified. Thus $M$
has one 3-cell, $n$ 2-cells, $n$ 1-cells, and one 0-cell, and since the boundaries of the $2$-cells spell the relators of $\mathcal{G}^{1/l}(n,0)$, it follows that $\mathcal{G}^{1/l}(n,0)$ is a spine of a closed 3-manifold, as required.

\begin{figure}
\begin{center}
\psset{xunit=0.55cm,yunit=0.5cm,algebraic=true,dimen=middle,dotstyle=o,dotsize=5pt 0,linewidth=1.6pt,arrowsize=3pt 2,arrowinset=0.25}
\psset{arrowscale=0.7,linewidth=1.0pt,dotsize=3pt}
\begin{pspicture*}(-10.7,-8.8)(16.3,12.75)
\psdots[dotstyle=*](2.,12.)\rput(2,12.5){$N$}
\psdots[dotstyle=*](0.,-8.)\rput(0,-8.5){$S$}
\begin{tiny}
\psdots[dotstyle=*](-8.,4.)\rput(-8.3,3.6){$u_{0}$}
\psdots[dotstyle=*](-6.,4.)\rput(-6.3,3.6){$v_{f}$}
\psdots[dotstyle=*](-4.,4.)\rput(-3.7,3.6){$u_{2}$}
\psdots[dotstyle=*](-2.,4.)\rput(-1.4,3.6){$v_{f+2}$}
\psdots[dotstyle=*](0.,4.)\rput(0.3,3.6){$u_{4}$}
\psdots[dotstyle=*](2.,4.)\rput(2.6,3.6){$v_{f+4}$}
\psdots[dotstyle=*](4.,4.)\rput(+4.4,3.6){$u_{6}$}
\psdots[dotstyle=*](6.,4.)\rput(6.6,3.6){$v_{f+6}$}
\psdots[dotstyle=*](8.,4.)\rput(+8.4,3.6){$u_{8}$}
\psdots[dotstyle=*](-5.5,4.)\rput(-5.6,4.5){$w_{f}^{1}$}
\psdots[dotstyle=*](-4.5,4.)\rput(-4.6,4.5){$w_{f}^{k-2}$}
\psdots[dotstyle=*](-1.5,4.)\rput(-1.6,4.5){$w_{f+2}^{1}$}
\psdots[dotstyle=*](-0.5,4.)\rput(-0.6,4.5){$w_{f+2}^{k-2}$}
\psdots[dotstyle=*](2.5,4.)\rput(2.4,4.5){$w_{f+4}^{1}$}
\psdots[dotstyle=*](3.5,4.)
\psdots[dotstyle=*](6.5,4.)\rput(6.4,4.5){$w_{f+6}^{1}$}
\psdots[dotstyle=*](7.5,4.)
\psdots[dotstyle=*](-10.,0.)\rput(-10,0.3){$u_{-1}$}
\psdots[dotstyle=*](-8.,0.)\rput(-7.4,0.3){$v_{f-1}$}
\psdots[dotstyle=*](-6.,0.)\rput(-5.7,0.3){$u_{1}$}
\psdots[dotstyle=*](-4.,0.)\rput(-3.4,0.3){$v_{f+1}$}
\psdots[dotstyle=*](-2.,0.)\rput(-1.7,0.3){$u_{3}$}
\psdots[dotstyle=*](0.,0.)\rput(0.6,0.3){$v_{f+3}$}
\psdots[dotstyle=*](2.,0.)\rput(2.3,0.3){$u_{5}$}
\psdots[dotstyle=*](4.,0.)\rput(4.6,0.3){$v_{f+5}$}
\psdots[dotstyle=*](6.,0.)\rput(6.3,0.3){$u_{7}$}
\psdots[dotstyle=*](8.,0.)\rput(8.6,0.3){$v_{f+7}$}
\psdots[dotstyle=*](10.,0.)\rput(10,0.3){$u_{9}$}
\psdots[dotstyle=*](-7.5,0.)\rput(-7.6,-0.5){$w_{f}^{1}$}
\psdots[dotstyle=*](-6.5,0.)\rput(-6.6,-0.5){$w_{f}^{k-2}$}
\psdots[dotstyle=*](-3.5,0.)\rput(-3.6,-0.5){$w_{f+2}^{1}$}
\psdots[dotstyle=*](-2.5,0.)\rput(-2.6,-0.5){$w_{f+2}^{k-2}$}
\psdots[dotstyle=*](0.5,0.)\rput(0.4,-0.5){$w_{f+4}^{1}$}
\psdots[dotstyle=*](1.5,0.)
\psdots[dotstyle=*](4.5,0.)\rput(4.4,-0.5){$w_{f+6}^{1}$}
\psdots[dotstyle=*](5.5,0.)
\psdots[dotstyle=*](8.5,0.)\rput(8.4,-0.5){$w_{f+8}^{1}$}
\psdots[dotstyle=*](9.5,0.)
\end{tiny}

\psline[linestyle=dotted](9,2)(10,2)
\psline[linestyle=dotted](-9,2)(-10,2)
\psline[linestyle=dotted](9,8)(10,8)
\psline[linestyle=dotted](-9,8)(-10,8)
\psline[linestyle=dotted](9,-4)(10,-4)
\psline[linestyle=dotted](-9,-4)(-10,-4)

\psline[ArrowInside=->](-8.,4.)(-6.,4.)
\psline[ArrowInside=->](-4.,4.)(-2.,4.)
\psline[ArrowInside=->](0.,4.)(2.,4.)
\psline[ArrowInside=->](4.,4.)(6.,4.)

\psline[ArrowInside=->](-6.,4.)(-5.5,4.)
\psline[linestyle=dotted](-5.3,4)(-4.7,4.)
\psline[ArrowInside=->](-4.5,4.)(-4.,4.)

\psline[ArrowInside=->](-2.,4.)(-1.5,4.)
\psline[linestyle=dotted](-1.3,4)(-0.7,4.)
\psline[ArrowInside=->](-0.5,4.)(0.,4.)

\psline[ArrowInside=->](2.,4.)(2.5,4.)
\psline[linestyle=dotted](2.7,4.)(3.3,4.)
\psline[ArrowInside=->](3.5,4.)(4.,4.)

\psline[ArrowInside=->](6.,4.)(6.5,4.)
\psline[linestyle=dotted](6.7,4.)(7.3,4.)
\psline[ArrowInside=->](7.5,4.)(8.,4.)

\psline[ArrowInside=->](-8.,0.)(-7.5,0.)
\psline[linestyle=dotted](-7.3,0)(-6.7,0.)
\psline[ArrowInside=->](-6.5,0.)(-6.,0.)

\psline[ArrowInside=->](-4.,0.)(-3.5,0.)
\psline[linestyle=dotted](-3.3,0)(-2.7,0.)
\psline[ArrowInside=->](-2.5,0.)(-2.,0.)

\psline[ArrowInside=->](0.,0.)(0.5,0.)
\psline[linestyle=dotted](0.7,0)(1.3,0.)
\psline[ArrowInside=->](1.5,0.)(2.,0.)

\psline[ArrowInside=->](4.,0.)(4.5,0.)
\psline[linestyle=dotted](4.7,0)(5.3,0.)
\psline[ArrowInside=->](5.5,0.)(6.,0.)

\psline[ArrowInside=->](8.,0.)(8.5,0.)
\psline[linestyle=dotted](8.7,0)(9.3,0.)
\psline[ArrowInside=->](9.5,0.)(10.,0.)

\psline[ArrowInside=->](-10.,0.)(-8.,0.)
\psline[ArrowInside=->](-6.,0.)(-4.,0.)
\psline[ArrowInside=->](-2.,0.)(0.,0.)
\psline[ArrowInside=->](2.,0.)(4.,0.)
\psline[ArrowInside=->](6.,0.)(8.,0.)
\psline[ArrowInside=->](-8.,4.)(-8.,0.)
\psline[ArrowInside=->](-6.,0.)(-6.,4.)
\psline[ArrowInside=->](-4.,4.)(-4.,0.)
\psline[ArrowInside=->](-2.,0.)(-2.,4.)
\psline[ArrowInside=->](0.,4.)(0.,0.)
\psline[ArrowInside=->](2.,0.)(2.,4.)
\psline[ArrowInside=->](4.,4.)(4.,0.)
\psline[ArrowInside=->](6.,0.)(6.,4.)
\psline[ArrowInside=->](8.,4.)(8.,0.)

\psline[ArrowInside=->](2.,12.)(-8.,4.)
\psline[ArrowInside=->](2.,12.)(-4,4.)
\psline[ArrowInside=->](2.,12.)(0.,4.)
\psline[ArrowInside=->](2.,12.)(4.,4.)
\psline[ArrowInside=->](2.,12.)(8.,4.)


\psline[ArrowInside=->](0.,-8.)(-10.,0.)
\psline[ArrowInside=->](0.,-8.)(-6.,0.)
\psline[ArrowInside=->](0.,-8.)(-2.,0.)
\psline[ArrowInside=->](0.,-8.)(2.,0.)
\psline[ArrowInside=->](0.,-8.)(6,0.)
\psline[ArrowInside=->](0.,-8.)(10.,0.)

\rput(-4.,5.7){$F_{0}^-$}
\rput(-1,5.7){$F_{2}^-$}
\rput(+2.0,5.7){$F_{4}^-$}
\rput(+5.2,5.7){$F_{6}^-$}
\rput(-7,2){$F_{f-1}^-$}
\rput(-5,2){$F_{f}^+$}
\rput(-3,2){$F_{f+1}^-$}
\rput(-1,2){$F_{f+2}^+$}
\rput(+1,2){$F_{f+3}^-$}
\rput(+3,2){$F_{f+4}^+$}
\rput(+5,2){$F_{f+5}^-$}
\rput(+7,2){$F_{f+6}^+$}
\rput(-6.2,-2.3){$F_{-1}^+$}
\rput(-3.2,-2.3){$F_{1}^+$}
\rput(+0.0,-2.3){$F_{3}^+$}
\rput(+3.2,-2.3){$F_{5}^+$}
\rput(+6.2,-2.3){$F_{7}^+$}
\end{pspicture*}
\end{center}
  \caption{Face pairing polyhedron for $\mathcal{G}^{k/1}(n,f)$\label{fig:Pk1(n)}}
\end{figure}


\noindent \textit{Case 2: $(k,l)=(5,2)$.}
Consider the face pairing polyhedron depicted in Figure \ref{fig:P52(n)} (where the arc labels can be deduced from the position of the unique source in the boundary of each face) with face pairing given by identifying the faces $F_i^+,F_i^-$ ($0\leq i<n)$. All the arcs labelled $x_0$ are contained in the following cycle:
\begin{alignat*}{1}
[N,s_0]
&\stackrel{F_0}{\longrightarrow} [u_{1-2f},t_0]
\stackrel{F_{-1}}{\longrightarrow} [w_{2f-1}^1,w_{2f-1}^2]
\stackrel{F_{2f-1}}{\longrightarrow} [r_{2f-1},v_{4f-1}]\\
&\stackrel{F_{4f-2}}{\longrightarrow} [s_{4f},u_{4f}]
\stackrel{F_{4f}}{\longrightarrow} [t_{4f},v_{4f}]
\stackrel{F_{4f-1}}{\longrightarrow} [w_{f-1}^2,u_{-f+1}]\\
&\stackrel{F_{f-1}}{\longrightarrow} [v_{3f-1},w_{3f-1}^1]
\stackrel{F_{3f-1}}{\longrightarrow} [u_{3f-1},r_{3f-1}]
\stackrel{F_{-2}}{\longrightarrow} [N,s_0],
\end{alignat*}
and all the arcs labelled $x_1$ are contained in the following cycle:
\begin{alignat*}{1}
[S,s_1]
&\stackrel{F_1}{\longrightarrow} [u_{2-2f},t_1]
\stackrel{F_{0}}{\longrightarrow} [w_{2f}^1,w_{2f}^2]
\stackrel{F_{2f}}{\longrightarrow} [r_{2f},v_{4f}]\\
&
\stackrel{F_{4f-1}}{\longrightarrow} [s_{4f+1},u_{4f+1}]
\stackrel{F_{4f+1}}{\longrightarrow} [t_{4f+1},v_{4f+1}]
\stackrel{F_{4f}}{\longrightarrow} [w_{f}^2,u_{-f+2}]\\
&
\stackrel{F_{f}}{\longrightarrow} [v_{3f},w_{3f}^1]
\stackrel{F_{3f}}{\longrightarrow} [u_{3f},r_{3f}]
\stackrel{F_{-1}}{\longrightarrow} [S,s_1].
\end{alignat*}
The proof then proceeds as in Case 1.

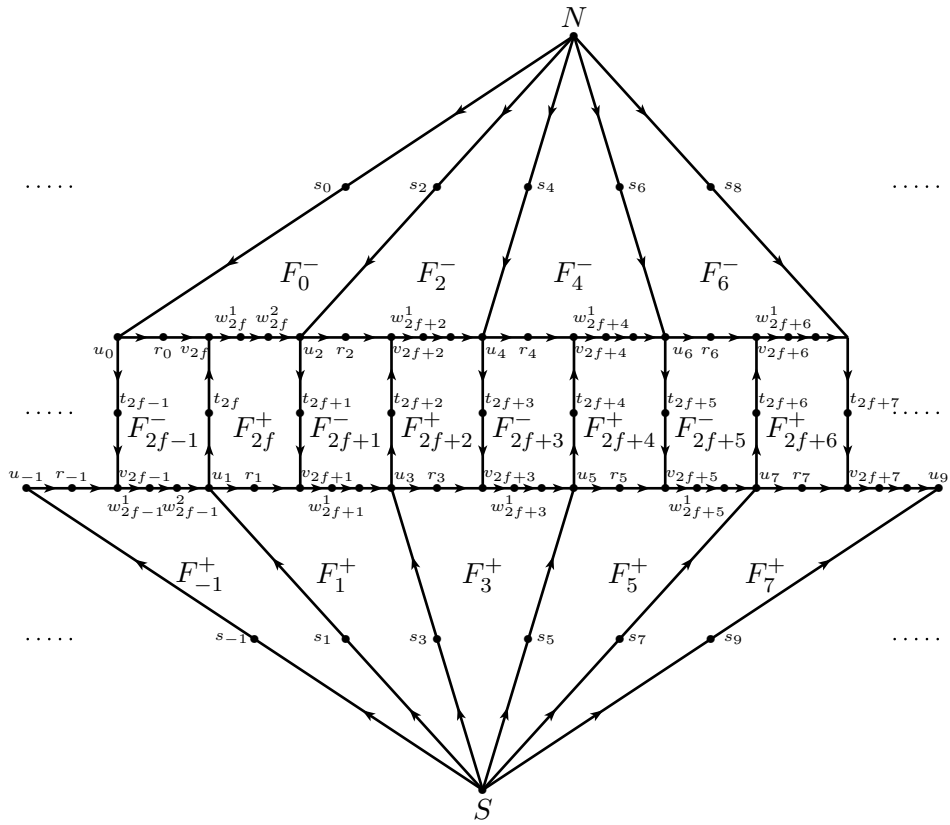
\begin{figure}
\begin{center}
\psset{xunit=0.55cm,yunit=0.5cm,algebraic=true,dimen=middle,dotstyle=o,dotsize=5pt 0,linewidth=1.6pt,arrowsize=3pt 2,arrowinset=0.25}
\psset{arrowscale=0.7,linewidth=1.0pt,dotsize=3pt}
\begin{pspicture*}(-10.7,-8.8)(16.3,12.75)
\psdots[dotstyle=*](2.,12.)\rput(2,12.5){$N$}
\psdots[dotstyle=*](0.,-8.)\rput(0,-8.5){$S$}
\begin{tiny}
\psdots[dotstyle=*](-3.,8.)\rput(-3.5,8){$s_{0}$}
\psdots[dotstyle=*](-1.,8.)\rput(-1.4,8){$s_{2}$}
\psdots[dotstyle=*](1.,8.)\rput(1.4,8){$s_{4}$}
\psdots[dotstyle=*](3.,8.)\rput(3.4,8){$s_{6}$}
\psdots[dotstyle=*](5,8.)\rput(5.5,8){$s_{8}$}
\psdots[dotstyle=*](-5.,-4.)\rput(-5.5,-4){$s_{-1}$}
\psdots[dotstyle=*](-3.,-4.)\rput(-3.5,-4){$s_{1}$}
\psdots[dotstyle=*](-1.,-4.)\rput(-1.4,-4){$s_{3}$}
\psdots[dotstyle=*](1.,-4.)\rput(1.4,-4){$s_{5}$}
\psdots[dotstyle=*](3.,-4.)\rput(3.4,-4){$s_{7}$}
\psdots[dotstyle=*](5,-4.)\rput(5.5,-4){$s_{9}$}
\psdots[dotstyle=*](-8.,4.)\rput(-8.3,3.6){$u_{0}$}
\psdots[dotstyle=*](-7.,4.)\rput(-7.0,3.6){$r_{0}$}
\psdots[dotstyle=*](-6.,4.)\rput(-6.3,3.6){$v_{2f}$}
\psdots[dotstyle=*](-4.,4.)\rput(-3.7,3.6){$u_{2}$}
\psdots[dotstyle=*](-3.,4.)\rput(-3.,3.6){$r_{2}$}
\psdots[dotstyle=*](-2.,4.)\rput(-1.4,3.6){$v_{2f+2}$}
\psdots[dotstyle=*](0.,4.)\rput(0.3,3.6){$u_{4}$}
\psdots[dotstyle=*](1.,4.)\rput(1,3.6){$r_{4}$}
\psdots[dotstyle=*](2.,4.)\rput(2.6,3.6){$v_{2f+4}$}
\psdots[dotstyle=*](4.,4.)\rput(+4.4,3.6){$u_{6}$}
\psdots[dotstyle=*](5.,4.)\rput(+5,3.6){$r_{6}$}
\psdots[dotstyle=*](6.,4.)\rput(6.6,3.6){$v_{2f+6}$}
\psdots[dotstyle=*](-5.3,4.)\rput(-5.5,4.5){$w_{2f}^{1}$}
\psdots[dotstyle=*](-4.7,4.)\rput(-4.6,4.5){$w_{2f}^{2}$}
\psdots[dotstyle=*](-1.3,4.)\rput(-1.4,4.5){$w_{2f+2}^{1}$}
\psdots[dotstyle=*](-0.7,4.)
\psdots[dotstyle=*](3.3,4.)\rput(2.6,4.5){$w_{2f+4}^{1}$}
\psdots[dotstyle=*](2.7,4.)
\psdots[dotstyle=*](6.7,4.)\rput(6.6,4.5){$w_{2f+6}^{1}$}
\psdots[dotstyle=*](7.3,4.)

\psdots[dotstyle=*](-10.,0.)\rput(-10,0.3){$u_{-1}$}
\psdots[dotstyle=*](-9.,0.)\rput(-9,0.3){$r_{-1}$}
\psdots[dotstyle=*](-8.,0.)\rput(-7.4,0.3){$v_{2f-1}$}
\psdots[dotstyle=*](-6.,0.)\rput(-5.7,0.3){$u_{1}$}
\psdots[dotstyle=*](-5.,0.)\rput(-5.,0.3){$r_{1}$}
\psdots[dotstyle=*](-4.,0.)\rput(-3.4,0.3){$v_{2f+1}$}
\psdots[dotstyle=*](-2.,0.)\rput(-1.7,0.3){$u_{3}$}
\psdots[dotstyle=*](-1.,0.)\rput(-1.,0.3){$r_{3}$}
\psdots[dotstyle=*](0.,0.)\rput(0.6,0.3){$v_{2f+3}$}
\psdots[dotstyle=*](2.,0.)\rput(2.3,0.3){$u_{5}$}
\psdots[dotstyle=*](3.,0.)\rput(3.,0.3){$r_{5}$}
\psdots[dotstyle=*](4.,0.)\rput(4.6,0.3){$v_{2f+5}$}
\psdots[dotstyle=*](6.,0.)\rput(6.3,0.3){$u_{7}$}
\psdots[dotstyle=*](7.,0.)\rput(7.,0.3){$r_{7}$}
\psdots[dotstyle=*](8.,0.)\rput(8.6,0.3){$v_{2f+7}$}
\psdots[dotstyle=*](10.,0.)\rput(10,0.3){$u_{9}$}
\psdots[dotstyle=*](-8.,2.)\rput(-7.4,2.3){$t_{2f-1}$}
\psdots[dotstyle=*](-6.,2.)\rput(-5.6,2.3){$t_{2f}$}
\psdots[dotstyle=*](-4.,2.)\rput(-3.4,2.3){$t_{2f+1}$}
\psdots[dotstyle=*](-2.,2.)\rput(-1.4,2.3){$t_{2f+2}$}
\psdots[dotstyle=*](0.,2.)\rput(0.6,2.3){$t_{2f+3}$}
\psdots[dotstyle=*](2.,2.)\rput(2.6,2.3){$t_{2f+4}$}
\psdots[dotstyle=*](4.,2.)\rput(4.6,2.3){$t_{2f+5}$}
\psdots[dotstyle=*](6.,2.)\rput(6.6,2.3){$t_{2f+6}$}
\psdots[dotstyle=*](8.,2.)\rput(8.6,2.3){$t_{2f+7}$}
\psdots[dotstyle=*](-7.3,0.)\rput(-7.6,-0.5){$w_{2f-1}^{1}$}
\psdots[dotstyle=*](-6.7,0.)\rput(-6.4,-0.5){$w_{2f-1}^{2}$}
\psdots[dotstyle=*](-3.3,0.)\rput(-3.2,-0.5){$w_{2f+1}^{1}$}
\psdots[dotstyle=*](-2.7,0.)
\psdots[dotstyle=*](0.7,0.)\rput(0.8,-0.5){$w_{2f+3}^{1}$}
\psdots[dotstyle=*](1.3,0.)
\psdots[dotstyle=*](4.7,0.)\rput(4.7,-0.5){$w_{2f+5}^{1}$}
\psdots[dotstyle=*](5.3,0.)
\psdots[dotstyle=*](8.7,0.)
\psdots[dotstyle=*](9.3,0.)
\end{tiny}

\psline[linestyle=dotted](9,2)(10,2)
\psline[linestyle=dotted](-9,2)(-10,2)
\psline[linestyle=dotted](9,8)(10,8)
\psline[linestyle=dotted](-9,8)(-10,8)
\psline[linestyle=dotted](9,-4)(10,-4)
\psline[linestyle=dotted](-9,-4)(-10,-4)

\psline[ArrowInside=->](-8.,4.)(-7.,4.)
\psline[ArrowInside=->](-7.,4.)(-6.,4.)
\psline[ArrowInside=->](-4.,4.)(-3.,4.)
\psline[ArrowInside=->](-3.,4.)(-2.,4.)
\psline[ArrowInside=->](0.,4.)(1.,4.)
\psline[ArrowInside=->](1.,4.)(2.,4.)
\psline[ArrowInside=->](4.,4.)(5.,4.)
\psline[ArrowInside=->](5.,4.)(6.,4.)

\psline[ArrowInside=->](-6.,4.)(-5.3,4.)
\psline[ArrowInside=->](-5.3,4)(-4.7,4.)
\psline[ArrowInside=->](-4.7,4.)(-4.,4.)

\psline[ArrowInside=->](-2.,4.)(-1.3,4.)
\psline[ArrowInside=->](-1.3,4)(-0.7,4.)
\psline[ArrowInside=->](-0.7,4.)(0.,4.)

\psline[ArrowInside=->](2.,4.)(2.7,4.)
\psline[ArrowInside=->](2.7,4.)(3.3,4.)
\psline[ArrowInside=->](3.3,4.)(4.,4.)

\psline[ArrowInside=->](6.,4.)(6.7,4.)
\psline[ArrowInside=->](6.7,4.)(7.3,4.)
\psline[ArrowInside=->](7.3,4.)(8.,4.)

\psline[ArrowInside=->](-8.,0.)(-7.3,0.)
\psline[ArrowInside=->](-7.3,0)(-6.7,0.)
\psline[ArrowInside=->](-6.7,0.)(-6.,0.)

\psline[ArrowInside=->](-4.,0.)(-3.3,0.)
\psline[ArrowInside=->](-3.3,0)(-2.7,0.)
\psline[ArrowInside=->](-2.7,0.)(-2.,0.)

\psline[ArrowInside=->](0.,0.)(0.7,0.)
\psline[ArrowInside=->](0.7,0)(1.3,0.)
\psline[ArrowInside=->](1.3,0.)(2.,0.)

\psline[ArrowInside=->](4.,0.)(4.7,0.)
\psline[ArrowInside=->](4.7,0)(5.3,0.)
\psline[ArrowInside=->](5.3,0.)(6.,0.)

\psline[ArrowInside=->](8.,0.)(8.7,0.)
\psline[ArrowInside=->](8.7,0)(9.3,0.)
\psline[ArrowInside=->](9.3,0.)(10.,0.)

\psline[ArrowInside=->](-10.,0.)(-9.,0.)
\psline[ArrowInside=->](-9.,0.)(-8.,0.)
\psline[ArrowInside=->](-6.,0.)(-5.,0.)
\psline[ArrowInside=->](-5.,0.)(-4.,0.)
\psline[ArrowInside=->](-2.,0.)(-1.,0.)
\psline[ArrowInside=->](-1.,0.)(0.,0.)
\psline[ArrowInside=->](2.,0.)(3.,0.)
\psline[ArrowInside=->](3.,0.)(4.,0.)
\psline[ArrowInside=->](6.,0.)(7.,0.)
\psline[ArrowInside=->](7.,0.)(8.,0.)
\psline[ArrowInside=->](-8.,4.)(-8.,2.)
\psline[ArrowInside=->](-8.,2.)(-8.,0.)
\psline[ArrowInside=->](-6.,0.)(-6.,2.)
\psline[ArrowInside=->](-6.,2.)(-6.,4.)
\psline[ArrowInside=->](-4.,4.)(-4.,2.)
\psline[ArrowInside=->](-4.,2.)(-4.,0.)
\psline[ArrowInside=->](-2.,0.)(-2.,2.)
\psline[ArrowInside=->](-2.,2.)(-2.,4.)
\psline[ArrowInside=->](0.,4.)(0.,2.)
\psline[ArrowInside=->](0.,2.)(0.,0.)
\psline[ArrowInside=->](2.,0.)(2.,2.)
\psline[ArrowInside=->](2.,2.)(2.,4.)
\psline[ArrowInside=->](4.,4.)(4.,2.)
\psline[ArrowInside=->](4.,2.)(4.,0.)
\psline[ArrowInside=->](6.,0.)(6.,2.)
\psline[ArrowInside=->](6.,2.)(6.,4.)
\psline[ArrowInside=->](8.,4.)(8.,2.)
\psline[ArrowInside=->](8.,2.)(8.,0.)

\psline[ArrowInside=->](2.,12.)(-3.,8.)
\psline[ArrowInside=->](2.,12.)(-1,8.)
\psline[ArrowInside=->](2.,12.)(1.,8.)
\psline[ArrowInside=->](2.,12.)(3.,8.)
\psline[ArrowInside=->](2.,12.)(5.,8.)
\psline[ArrowInside=->](-3.,8.)(-8.,4.)
\psline[ArrowInside=->](-1.,8.)(-4,4.)
\psline[ArrowInside=->](1.,8.)(0.,4.)
\psline[ArrowInside=->](3.,8.)(4.,4.)
\psline[ArrowInside=->](5.,8.)(8.,4.)


\psline[ArrowInside=->](0.,-8.)(-5.,-4.)
\psline[ArrowInside=->](0.,-8.)(-3.,-4.)
\psline[ArrowInside=->](0.,-8.)(-1.,-4.)
\psline[ArrowInside=->](0.,-8.)(1.,-4.)
\psline[ArrowInside=->](0.,-8.)(3,-4.)
\psline[ArrowInside=->](0.,-8.)(5,-4.)
\psline[ArrowInside=->](-5.,-4.)(-10.,0.)
\psline[ArrowInside=->](-3.,-4.)(-6.,0.)
\psline[ArrowInside=->](-1.,-4.)(-2.,0.)
\psline[ArrowInside=->](1.,-4.)(2.,0.)
\psline[ArrowInside=->](3.,-4.)(6.,0.)
\psline[ArrowInside=->](5.,-4.)(10,0.)

\rput(-4.,5.7){$F_{0}^-$}
\rput(-1,5.7){$F_{2}^-$}
\rput(+2.0,5.7){$F_{4}^-$}
\rput(+5.2,5.7){$F_{6}^-$}
\rput(-7,1.5){$F_{2f-1}^-$}
\rput(-5,1.5){$F_{2f}^+$}
\rput(-3,1.5){$F_{2f+1}^-$}
\rput(-1,1.5){$F_{2f+2}^+$}
\rput(+1,1.5){$F_{2f+3}^-$}
\rput(+3,1.5){$F_{2f+4}^+$}
\rput(+5,1.5){$F_{2f+5}^-$}
\rput(+7,1.5){$F_{2f+6}^+$}
\rput(-6.2,-2.3){$F_{-1}^+$}
\rput(-3.2,-2.3){$F_{1}^+$}
\rput(+0.0,-2.3){$F_{3}^+$}
\rput(+3.2,-2.3){$F_{5}^+$}
\rput(+6.2,-2.3){$F_{7}^+$}
\end{pspicture*}
\end{center}
  \caption{Face pairing polyhedron for $\mathcal{G}^{5/2}(n,f)$\label{fig:P52(n)}}
\end{figure}


\noindent \textit{Case 3: $(k,l)=(1,l)$.}
Consider the face pairing polyhedron depicted in Figure \ref{fig:P1l(n)}. All the arcs labelled $x_0$ are contained in the following cycle:
\begin{alignat*}{1}
[N,u_0^{1}]
&\stackrel{F_0}{\longrightarrow} [w_0,t_0^1]
\stackrel{F_{-2}}{\longrightarrow} [u_0^1,u_0^2]
\stackrel{F_{0}}{\longrightarrow} [t_0^1,t_0^2]
\stackrel{F_{-2}}{\longrightarrow} [u_0^2,u_0^3]\\
&\stackrel{F_{0}}{\longrightarrow} [t_0^2,t_0^3]
\stackrel{}{\longrightarrow} \cdots
\stackrel{F_{-2}}{\longrightarrow} [u_0^{l-2},u_0^{l-1}]
\stackrel{F_{0}}{\longrightarrow} [t_0^{l-2},v_{-1}]\\
&\stackrel{F_{-2}}{\longrightarrow} [u_0^{l-1},w_{-1}]
\stackrel{F_{0}}{\longrightarrow} [v_{-1},v_0]
\stackrel{F_{-1}}{\longrightarrow} [w_{-2},w_0]
\stackrel{F_{-2}}{\longrightarrow} [N,u_0^1],
\end{alignat*}
and all the arcs labelled $x_1$ are contained in the following cycle:
\begin{alignat*}{1}
[S,u_1^{1}]
&\stackrel{F_1}{\longrightarrow} [w_1,t_1^1]
\stackrel{F_{-1}}{\longrightarrow} [u_1^1,u_1^2]
\stackrel{F_{1}}{\longrightarrow} [t_1^1,t_1^2]
\stackrel{F_{-1}}{\longrightarrow} [u_1^2,u_1^3]\\
&\stackrel{F_{1}}{\longrightarrow} [t_1^2,t_1^3]
\stackrel{}{\longrightarrow} \cdots
\stackrel{F_{-1}}{\longrightarrow} [u_1^{l-2},u_1^{l-1}]
\stackrel{F_{1}}{\longrightarrow} [t_1^{l-2},v_{0}]\\
&\stackrel{F_{-1}}{\longrightarrow} [u_1^{l-1},w_{0}]
\stackrel{F_{1}}{\longrightarrow} [v_{0},v_1]
\stackrel{F_{0}}{\longrightarrow} [w_{-1},w_1]
\stackrel{F_{-1}}{\longrightarrow} [S,u_1^1].
\end{alignat*}
The proof then proceeds as in Case 1.

\begin{figure}

\begin{center}
\psset{xunit=0.55cm,yunit=0.7cm,algebraic=true,dimen=middle,dotstyle=o,dotsize=5pt 0,linewidth=1.6pt,arrowsize=3pt 2,arrowinset=0.25}
\psset{arrowscale=0.7,linewidth=1.0pt,dotsize=3pt}
\begin{pspicture*}(-11,-13)(11.5,14)
\psdots[dotstyle=*](2.,12.)\rput(2,12.5){$N$}
\psdots[dotstyle=*](0.,-12.)\rput(0,-12.5){$S$}
\begin{scriptsize}
\psdots[dotstyle=*](-0.48,10.)\rput(-1.2,10){$u_{0}^{1}$}
\psdots[dotstyle=*](0.48,10.)
\psdots[dotstyle=*](1.5,10.)
\psdots[dotstyle=*](2.5,10.)
\psdots[dotstyle=*](3.52,10.)\rput(4.1,10){$u_{8}^{1}$}
\psdots[dotstyle=*](-5.6,6.)\rput(-5.9,6.2){$u_{0}^{l-1}$}
\psdots[dotstyle=*](-2.52,6.)\rput(-2.9,6.2){$u_{2}^{l-1}$}
\psdots[dotstyle=*](0.44,6.)\rput(1.1,6.2){$u_{4}^{l-1}$}
\psdots[dotstyle=*](3.56,6.)\rput(4.2,6.2){$u_{6}^{l-1}$}
\psdots[dotstyle=*](6.52,6.)\rput(7.2,6.2){$u_{8}^{l-1}$}
\psdots[dotstyle=*](-8.,4.)\rput(-8.3,4.3){$w_{-1}$}
\psdots[dotstyle=*](-4.,4.)\rput(-4.3,4.3){$w_{1}$}
\psdots[dotstyle=*](0.,4.)\rput(-0.3,4.3){$w_{3}$}
\psdots[dotstyle=*](4.,4.)\rput(+4.3,4.3){$w_{5}$}
\psdots[dotstyle=*](8.,4.)\rput(+8.3,4.3){$w_{7}$}
\psdots[dotstyle=*](-8.,3.)\rput(-7.4,3){$t_{-1}^1$}
\psdots[dotstyle=*](-4.,3.)\rput(-3.5,3){$t_{1}^1$}
\psdots[dotstyle=*](0.,3.)\rput(0.5,3){$t_{3}^1$}
\psdots[dotstyle=*](4.,3.)\rput(4.5,3){$t_{5}^1$}
\psdots[dotstyle=*](8.,3.)\rput(8.5,3){$t_{7}^1$}
\psdots[dotstyle=*](-8.,1.)\rput(-7.4,1){$t_{-1}^{l-2}$}
\psdots[dotstyle=*](-4.,1.)\rput(-3.3,1){$t_{1}^{l-2}$}
\psdots[dotstyle=*](0.,1.)\rput(0.7,1){$t_{3}^{l-2}$}
\psdots[dotstyle=*](4.,1.)\rput(4.7,1){$t_{5}^{l-2}$}
\psdots[dotstyle=*](8.,1.)\rput(8.7,1){$t_{7}^{l-2}$}
\psdots[dotstyle=*](-10.,0.)\rput(-10,0.3){$v_{-3}$}
\psdots[dotstyle=*](-8.,0.)\rput(-8,-0.3){$v_{-2}$}
\psdots[dotstyle=*](-6.,0.)\rput(-6,0.3){$v_{-1}$}
\psdots[dotstyle=*](-4.,0.)\rput(-4,-0.3){$v_{0}$}
\psdots[dotstyle=*](-2.,0.)\rput(-2,0.3){$v_{1}$}
\psdots[dotstyle=*](0.,0.)\rput(-0,-0.3){$v_{2}$}
\psdots[dotstyle=*](2.,0.)\rput(2,0.3){$v_{3}$}
\psdots[dotstyle=*](4.,0.)\rput(4,-0.3){$v_{4}$}
\psdots[dotstyle=*](6.,0.)\rput(6,0.3){$v_{5}$}
\psdots[dotstyle=*](8.,0.)\rput(8,-0.3){$v_{6}$}
\psdots[dotstyle=*](10.,0.)\rput(10,0.3){$v_{7}$}
\psdots[dotstyle=*](-10.,-1.)\rput(-9.3,-1.0){$t_{-2}^{l-2}$}
\psdots[dotstyle=*](-6.,-1.)\rput(-5.3,-1.0){$t_{0}^{l-2}$}
\psdots[dotstyle=*](-2.,-1.)\rput(-1.3,-1.0){$t_{2}^{l-2}$}
\psdots[dotstyle=*](2.,-1.)\rput(2.7,-1.0){$t_{4}^{l-2}$}
\psdots[dotstyle=*](6.,-1.)\rput(6.7,-1.0){$t_{6}^{l-2}$}
\psdots[dotstyle=*](10.,-1.)\rput(9.4,-1.0){$t_{8}^{l-2}$}
\psdots[dotstyle=*](-10.,-3.)\rput(-9.3,-3.0){$t_{-2}^{1}$}
\psdots[dotstyle=*](-6.,-3.)\rput(-5.5,-3.0){$t_{0}^{1}$}
\psdots[dotstyle=*](-2.,-3.)\rput(-1.5,-3.0){$t_{2}^{1}$}
\psdots[dotstyle=*](2.,-3.)\rput(2.5,-3.0){$t_{4}^{1}$}
\psdots[dotstyle=*](6.,-3.)\rput(6.5,-3.0){$t_{6}^{1}$}
\psdots[dotstyle=*](10.,-3.)\rput(9.6,-3.0){$t_{8}^{1}$}
\psdots[dotstyle=*](-10.,-4.)\rput(-10.3,-4.3){$w_{-2}$}
\psdots[dotstyle=*](-6.,-4.)\rput(-6.3,-4.3){$w_{0}$}
\psdots[dotstyle=*](-2.,-4.)\rput(-2.4,-4.3){$w_{2}$}
\psdots[dotstyle=*](2.,-4.)\rput(2.3,-4.3){$w_{4}$}
\psdots[dotstyle=*](6.,-4.)\rput(6.3,-4.3){$w_{6}$}
\psdots[dotstyle=*](10.,-4.)\rput(9.2,-4.2){$w_{8}$}
\psdots[dotstyle=*](-7.56,-6.)\rput(-8.4,-6){$u_{-1}^{l-1}$}
\psdots[dotstyle=*](-4.52,-6.)\rput(-5.2,-6){$u_{1}^{l-1}$}
\psdots[dotstyle=*](-1.56,-6.)\rput(-2.2,-6){$u_{3}^{l-1}$}
\psdots[dotstyle=*](1.56,-6.)\rput(2.3,-6){$u_{5}^{l-1}$}
\psdots[dotstyle=*](4.52,-6.)\rput(5.3,-6){$u_{7}^{l-1}$}
\psdots[dotstyle=*](7.56,-6.)\rput(8.4,-6){$u_{9}^{l-1}$}
\psdots[dotstyle=*](-2.48,-10.)\rput(-3.0,-10){$u_{-1}^{1}$}
\psdots[dotstyle=*](-1.52,-10.)
\psdots[dotstyle=*](-0.5,-10.)
\psdots[dotstyle=*](0.5,-10.)
\psdots[dotstyle=*](1.52,-10.)
\psdots[dotstyle=*](2.48,-10.)\rput(3.1,-10){$u_{9}^{1}$}
\end{scriptsize}
\psline[linestyle=dotted](9,2)(10,2)
\psline[linestyle=dotted](-9,2)(-10,2)
\psline[linestyle=dotted](9,-8)(10,-8)
\psline[linestyle=dotted](-9,-8)(-10,-8)
\psline[linestyle=dotted](9,8)(10,8)
\psline[linestyle=dotted](-9,8)(-10,8)

\psline[ArrowInside=->](-8.,4.)(-4.,4.)
\psline[ArrowInside=->](-4.,4.)(0.,4.)
\psline[ArrowInside=->](0.,4.)(4.,4.)
\psline[ArrowInside=->](4.,4.)(8.,4.)
\psline[ArrowInside=->](-10.,0.)(-8.,0.)
\psline[ArrowInside=->](-8.,0.)(-6.,0.)
\psline[ArrowInside=->](-6.,0.)(-4.,0.)
\psline[ArrowInside=->](-4.,0.)(-2.,0.)
\psline[ArrowInside=->](-2.,0.)(0.,0.)
\psline[ArrowInside=->](0.,0.)(2.,0.)
\psline[ArrowInside=->](2.,0.)(4.,0.)
\psline[ArrowInside=->](4.,0.)(6.,0.)
\psline[ArrowInside=->](6.,0.)(8.,0.)
\psline[ArrowInside=->](8.,0.)(10.,0.)
\psline[ArrowInside=->](-10.,-4.)(-6.,-4.)
\psline[ArrowInside=->](-6.,-4.)(-2.,-4.)
\psline[ArrowInside=->](-2.,-4.)(2.,-4.)
\psline[ArrowInside=->](2.,-4.)(6.,-4.)
\psline[ArrowInside=->](6.,-4.)(10.,-4.)
\psline[ArrowInside=->](-8.,4.)(-8.,3.)
\psline[linestyle=dotted](-8,2.5)(-8,1.5)
\psline[ArrowInside=->](-8.,1.)(-8.,0.)
\psline[ArrowInside=->](-4.,4.)(-4.,3.)
\psline[linestyle=dotted](-4,2.5)(-4,1.5)
\psline[ArrowInside=->](-4.,1.)(-4.,0.)
\psline[ArrowInside=->](0.,1.)(0.,0.)
\psline[ArrowInside=->](0.,4.)(0.,3.)
\psline[linestyle=dotted](0,2.5)(0,1.5)
\psline[ArrowInside=->](0.,1.)(0.,0.)
\psline[ArrowInside=->](4.,4.)(4.,3.)
\psline[linestyle=dotted](4,2.5)(4,1.5)
\psline[ArrowInside=->](4.,1.)(4.,0.)
\psline[ArrowInside=->](8.,4.)(8.,3.)
\psline[linestyle=dotted](8,2.5)(8,1.5)
\psline[ArrowInside=->](8.,1.)(8.,0.)
\psline[ArrowInside=->](10.,-4.)(10.,-3.)
\psline[linestyle=dotted](10,-2.5)(10,-1.5)
\psline[ArrowInside=->](10.,-1.)(10.,0.)
\psline[ArrowInside=->](-10.,-4.)(-10.,-3.)
\psline[linestyle=dotted](-10,-2.5)(-10,-1.5)
\psline[ArrowInside=->](-10.,-1.)(-10.,0.)
\psline[ArrowInside=->](-6.,-4.)(-6.,-3.)
\psline[linestyle=dotted](-6,-2.5)(-6,-1.5)
\psline[ArrowInside=->](-6.,-1.)(-6.,0.)
\psline[ArrowInside=->](-2.,-1.)(-2.,0.)
\psline[linestyle=dotted](-2,-2.5)(-2,-1.5)
\psline[ArrowInside=->](-2.,-4.)(-2.,-3.)
\psline[ArrowInside=->](2.,-4.)(2.,-3.)
\psline[linestyle=dotted](2,-2.5)(2,-1.5)
\psline[ArrowInside=->](2.,-1.)(2.,0.)
\psline[ArrowInside=->](6.,-4.)(6.,-3.)
\psline[linestyle=dotted](6,-2.5)(6,-1.5)
\psline[ArrowInside=->](6.,-1.)(6.,0.)
\psline[ArrowInside=->](2.,12.)(-0.48,10.)
\psline[ArrowInside=->](2.,12.)(0.48,10.)
\psline[ArrowInside=->](2.,12.)(1.5,10.)
\psline[ArrowInside=->](2.,12.)(2.5,10.)
\psline[ArrowInside=->](2.,12.)(3.52,10.)

\psline[linestyle=dotted](-0.48,10.)(-5.56,6.)
\psline[linestyle=dotted](0.48,10.)(-2.52,6.)
\psline[linestyle=dotted](1.5,10.)(0.44,6.)
\psline[linestyle=dotted](2.5,10.)(3.56,6.)
\psline[linestyle=dotted](3.52,10.)(6.52,6.)

\psline[ArrowInside=->](-5.56,6.)(-8.,4.)
\psline[ArrowInside=->](-2.52,6.)(-4.,4.)
\psline[ArrowInside=->](0.44,6.)(0.,4.)
\psline[ArrowInside=->](3.56,6.)(4.,4.)
\psline[ArrowInside=->](6.52,6.)(8.,4.)

\psline[ArrowInside=->](0.,-12.)(-2.48,-10.)
\psline[ArrowInside=->](0.,-12.)(-1.52,-10.)
\psline[ArrowInside=->](0.,-12.)(-0.5,-10.)
\psline[ArrowInside=->](0.,-12.)(0.5,-10.)
\psline[ArrowInside=->](0.,-12.)(1.52,-10.)
\psline[ArrowInside=->](0.,-12.)(2.48,-10.)

\psline[linestyle=dotted](-2.48,-10.)(-7.56,-6.)
\psline[linestyle=dotted](-1.52,-10.)(-4.52,-6.)
\psline[linestyle=dotted](-0.5,-10.)(-1.56,-6.)
\psline[linestyle=dotted](0.5,-10.)(1.56,-6.)
\psline[linestyle=dotted](1.52,-10.)(4.52,-6.)
\psline[linestyle=dotted](2.48,-10.)(7.56,-6.)

\psline[ArrowInside=->](-7.56,-6.)(-10.,-4.)
\psline[ArrowInside=->](-4.52,-6.)(-6.,-4.)
\psline[ArrowInside=->](-1.56,-6.)(-2.,-4.)
\psline[ArrowInside=->](1.56,-6.)(2.,-4.)
\psline[ArrowInside=->](4.52,-6.)(6.,-4.)
\psline[ArrowInside=->](7.56,-6.)(10.,-4.)
%
\rput(-4.,5.7){$F_{0}^-$}
\rput(-1,5.7){$F_{2}^-$}
\rput(+2.0,5.7){$F_{4}^-$}
\rput(+5.2,5.7){$F_{6}^-$}
\rput(-6,2){$F_{-1}^-$}
\rput(-2,2){$F_{1}^-$}
\rput(+2,2){$F_{3}^-$}
\rput(+6,2){$F_{5}^-$}
\rput(-8,-2){$F_{-2}^+$}
\rput(-4,-2){$F_{0}^+$}
\rput(0,-2){$F_{2}^+$}
\rput(+4,-2){$F_{4}^+$}
\rput(+8,-2){$F_{6}^+$}
\rput(-6.2,-6.3){$F_{-1}^+$}
\rput(-3.2,-6.3){$F_{1}^+$}
\rput(+0.0,-6.3){$F_{3}^+$}
\rput(+3.2,-6.3){$F_{5}^+$}
\rput(+6.2,-6.3){$F_{7}^+$}
\end{pspicture*}
\end{center}
  \caption{Face pairing polyhedron for $\mathcal{G}^{1/l}(n,0)=\mathcal{F}^{1/l}(n)$\label{fig:P1l(n)}}
\end{figure}
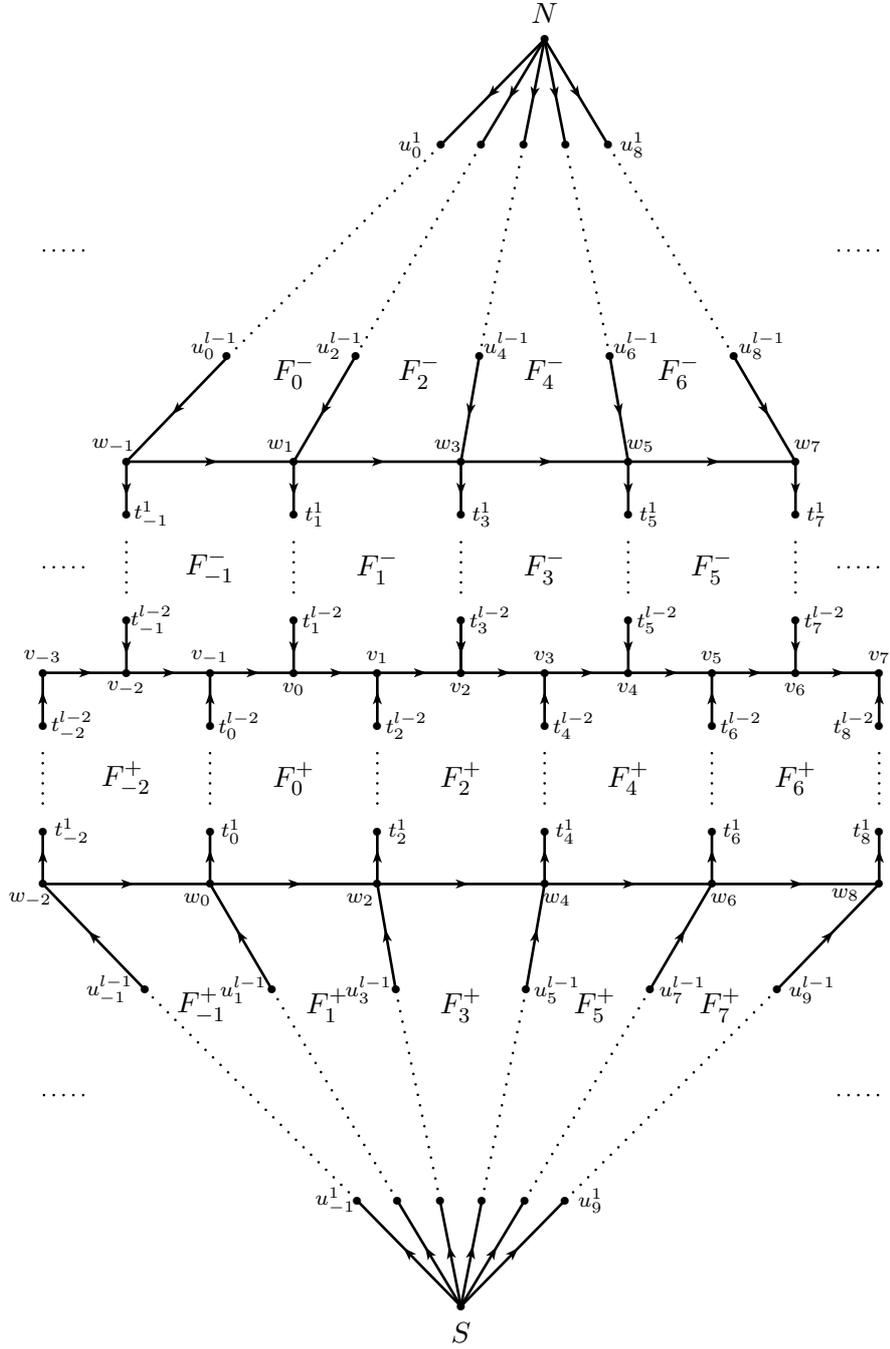


\noindent \textit{Case 4: $(k,l)=(2,5)$.}
Consider the face pairing polyhedron depicted in Figure \ref{fig:P25(n)}. All the arcs labelled $x_0$ are contained in the following cycle:
\begin{alignat*}{1}
[N,u_0^{1}]
&\stackrel{F_0}{\longrightarrow} [w_0,t_0^1]
\stackrel{F_{-2}}{\longrightarrow} [u_{0}^2,u_{0}^3]
\stackrel{F_{0}}{\longrightarrow} [t_0^2,v_{-1}]
\stackrel{F_{-2}}{\longrightarrow} [u_0^4,w_{f-1}]\\
&\stackrel{F_{0}}{\longrightarrow} [s_{-1},v_0]
\stackrel{F_{-1}}{\longrightarrow} [r_{f-2},w_f]
\stackrel{F_{f-2}}{\longrightarrow} [u_f^{1},u_{f}^2]
\stackrel{F_{f}}{\longrightarrow} [t_f^{1},t_{f}^2]\\
&
\stackrel{F_{f-2}}{\longrightarrow} [u_{f}^3,u_f^4]
\stackrel{F_{f}}{\longrightarrow} [v_{f-1},s_{f-1}]
\stackrel{F_{f-1}}{\longrightarrow} [w_{2f-2},r_{2f-2}]
\stackrel{F_{2f-2}}{\longrightarrow} [N,u_{2f}^1],
\end{alignat*}
and all the arcs labelled $x_1$ are contained in the following cycle:
\begin{alignat*}{1}
[S,u_1^{1}]
&\stackrel{F_1}{\longrightarrow} [w_1,t_1^1]
\stackrel{F_{-1}}{\longrightarrow} [u_{1}^2,u_{1}^3]
\stackrel{F_{1}}{\longrightarrow} [t_1^2,v_{0}]
\stackrel{F_{-1}}{\longrightarrow} [u_1^4,w_{f}]\\
&\stackrel{F_{1}}{\longrightarrow} [s_{0},v_1]
\stackrel{F_{0}}{\longrightarrow} [r_{f-1},w_{f+1}]
\stackrel{F_{f-1}}{\longrightarrow} [u_{f+1}^{1},u_{f+1}^2]
\stackrel{F_{f+1}}{\longrightarrow} [t_{f+1}^{1},t_{f+1}^2]\\
&
\stackrel{F_{f-1}}{\longrightarrow} [u_{f+1}^3,u_{f+1}^4]
\stackrel{F_{f+1}}{\longrightarrow} [v_{f},s_{f}]
\stackrel{F_{f}}{\longrightarrow} [w_{2f-1},r_{2f-1}]
\stackrel{F_{2f-1}}{\longrightarrow} [S,u_{2f+1}^1].
\end{alignat*}
The proof then proceeds as in Case 1.
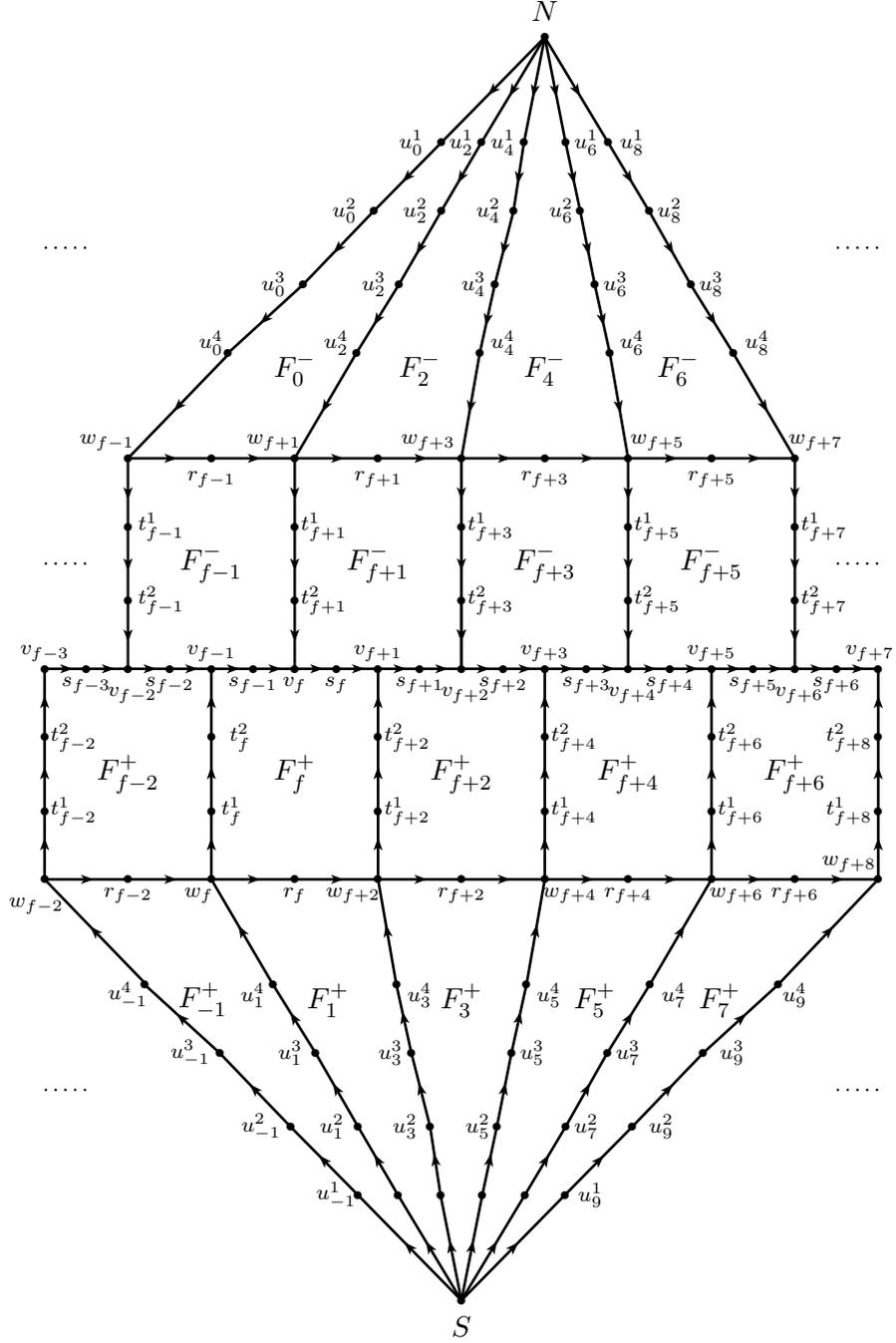
\begin{figure}
    \begin{center}
\psset{xunit=0.55cm,yunit=0.7cm,algebraic=true,dimen=middle,dotstyle=o,dotsize=5pt 0,linewidth=1.6pt,arrowsize=3pt 2,arrowinset=0.25}
\psset{arrowscale=0.7,linewidth=1.0pt,dotsize=3pt}
\begin{pspicture*}(-11,-13)(11.5,14)
\psdots[dotstyle=*](2.,12.)\rput(2,12.5){$N$}
\psdots[dotstyle=*](0.,-12.)\rput(0,-12.5){$S$}
\begin{scriptsize}
\psdots[dotstyle=*](-0.48,10.)\rput(-1.2,10){$u_{0}^{1}$}
\psdots[dotstyle=*](0.48,10.)\rput(-0.,10){$u_{2}^{1}$}
\psdots[dotstyle=*](1.5,10.)\rput(1.0,10){$u_{4}^{1}$}
\psdots[dotstyle=*](2.5,10.)\rput(3.0,10){$u_{6}^{1}$}
\psdots[dotstyle=*](3.52,10.)\rput(4.1,10){$u_{8}^{1}$}
\psdots[dotstyle=*](-2.1,8.7)\rput(-2.8,8.7){$u_{0}^{2}$}
\psdots[dotstyle=*](-0.48,8.7)\rput(-1.1,8.7){$u_{2}^{2}$}
\psdots[dotstyle=*](1.25,8.7)\rput(0.65,8.7){$u_{4}^{2}$}
\psdots[dotstyle=*](2.85,8.7)\rput(2.4,8.7){$u_{6}^{2}$}
\psdots[dotstyle=*](4.5,8.7)\rput(5.0,8.7){$u_{8}^{2}$}
\psline[ArrowInside=->](-0.48,10.)(-2.1,8.7)
\psline[ArrowInside=->](0.48,10.)(-0.48,8.7)
\psline[ArrowInside=->](1.5,10.)(1.25,8.7)
\psline[ArrowInside=->](2.5,10.)(2.85,8.7)
\psline[ArrowInside=->](3.52,10.)(4.5,8.7)
\psline[ArrowInside=->](-2.1,8.7)(-3.8,7.3)
\psline[ArrowInside=->](-0.48,8.7)(-1.5,7.3)
\psline[ArrowInside=->](1.25,8.7)(0.8,7.3)
\psline[ArrowInside=->](2.85,8.7)(3.2,7.3)
\psline[ArrowInside=->](4.5,8.7)(5.5,7.3)
\psline[ArrowInside=->](-3.8,7.3)(-5.6,6.)
\psline[ArrowInside=->](-1.5,7.3)(-2.52,6.)
\psline[ArrowInside=->](0.8,7.3)(0.44,6.)
\psline[ArrowInside=->](3.2,7.3)(3.56,6.)
\psline[ArrowInside=->](5.5,7.3)(6.52,6.)
%
\psdots[dotstyle=*](-3.8,7.3)\rput(-4.5,7.3){$u_{0}^{3}$}
\psdots[dotstyle=*](-1.5,7.3)\rput(-2.1,7.3){$u_{2}^{3}$}
\psdots[dotstyle=*](0.8,7.3)\rput(0.3,7.3){$u_{4}^{3}$}
\psdots[dotstyle=*](3.2,7.3)\rput(3.7,7.3){$u_{6}^{3}$}
\psdots[dotstyle=*](5.5,7.3)\rput(6.0,7.3){$u_{8}^{3}$}
\psdots[dotstyle=*](-5.6,6.)\rput(-6.0,6.2){$u_{0}^{4}$}
\psdots[dotstyle=*](-2.52,6.)\rput(-3.0,6.2){$u_{2}^{4}$}
\psdots[dotstyle=*](0.44,6.)\rput(1.0,6.2){$u_{4}^{4}$}
\psdots[dotstyle=*](3.56,6.)\rput(4.1,6.2){$u_{6}^{4}$}
\psdots[dotstyle=*](6.52,6.)\rput(7.1,6.2){$u_{8}^{4}$}
\psdots[dotstyle=*](-8.,4.)\rput(-8.5,4.3){$w_{f-1}$}
\psdots[dotstyle=*](-6.,4.)\rput(-6.0,3.6){$r_{f-1}$}
\psdots[dotstyle=*](-4.,4.)\rput(-4.5,4.3){$w_{f+1}$}
\psdots[dotstyle=*](-2.,4.)\rput(-2.0,3.6){$r_{f+1}$}
\psdots[dotstyle=*](0.,4.)\rput(-0.8,4.3){$w_{f+3}$}
\psdots[dotstyle=*](2.,4.)\rput(2.0,3.6){$r_{f+3}$}
\psdots[dotstyle=*](4.,4.)\rput(+4.7,4.3){$w_{f+5}$}
\psdots[dotstyle=*](6.,4.)\rput(6.0,3.6){$r_{f+5}$}
\psdots[dotstyle=*](8.,4.)\rput(+8.5,4.3){$w_{f+7}$}
\psdots[dotstyle=*](-8.,2.7)\rput(-7.2,2.7){$t_{f-1}^1$}
\psdots[dotstyle=*](-4.,2.7)\rput(-3.3,2.7){$t_{f+1}^1$}
\psdots[dotstyle=*](0.,2.7)\rput(0.7,2.7){$t_{f+3}^1$}
\psdots[dotstyle=*](4.,2.7)\rput(4.7,2.7){$t_{f+5}^1$}
\psdots[dotstyle=*](8.,2.7)\rput(8.7,2.7){$t_{f+7}^1$}

\psdots[dotstyle=*](-8.,1.3)\rput(-7.2,1.3){$t_{f-1}^{2}$}
\psdots[dotstyle=*](-4.,1.3)\rput(-3.3,1.3){$t_{f+1}^{2}$}
\psdots[dotstyle=*](0.,1.3)\rput(0.7,1.3){$t_{f+3}^{2}$}
\psdots[dotstyle=*](4.,1.3)\rput(4.7,1.3){$t_{f+5}^{2}$}
\psdots[dotstyle=*](8.,1.3)\rput(8.7,1.3){$t_{f+7}^{2}$}
\psdots[dotstyle=*](-10.,0.)\rput(-10,0.3){$v_{f-3}$}
\psdots[dotstyle=*](-9.,0.)\rput(-9,-0.3){$s_{f-3}$}
\psdots[dotstyle=*](-8.,0.)\rput(-7.9,-0.5){$v_{f-2}$}
\psdots[dotstyle=*](-7.,0.)\rput(-7,-0.3){$s_{f-2}$}
\psdots[dotstyle=*](-6.,0.)\rput(-6,0.3){$v_{f-1}$}
\psdots[dotstyle=*](-5.,0.)\rput(-5,-0.3){$s_{f-1}$}
\psdots[dotstyle=*](-4.,0.)\rput(-4,-0.3){$v_{f}$}
\psdots[dotstyle=*](-3.,0.)\rput(-3,-0.3){$s_{f}$}
\psdots[dotstyle=*](-2.,0.)\rput(-2,0.3){$v_{f+1}$}
\psdots[dotstyle=*](-1.,0.)\rput(-1,-0.3){$s_{f+1}$}
\psdots[dotstyle=*](0.,0.)\rput(0.1,-0.5){$v_{f+2}$}
\psdots[dotstyle=*](1.,0.)\rput(1,-0.3){$s_{f+2}$}
\psdots[dotstyle=*](2.,0.)\rput(2,0.3){$v_{f+3}$}
\psdots[dotstyle=*](3.,0.)\rput(3,-0.3){$s_{f+3}$}
\psdots[dotstyle=*](4.,0.)\rput(4.1,-0.5){$v_{f+4}$}
\psdots[dotstyle=*](5.,0.)\rput(5,-0.3){$s_{f+4}$}
\psdots[dotstyle=*](6.,0.)\rput(6,0.3){$v_{f+5}$}
\psdots[dotstyle=*](7.,0.)\rput(7,-0.3){$s_{f+5}$}
\psdots[dotstyle=*](8.,0.)\rput(8.1,-0.5){$v_{f+6}$}
\psdots[dotstyle=*](9.,0.)\rput(9,-0.3){$s_{f+6}$}
\psdots[dotstyle=*](10.,0.)\rput(9.8,0.3){$v_{f+7}$}
\psdots[dotstyle=*](-10.,-1.3)\rput(-9.3,-1.3){$t_{f-2}^{2}$}
\psdots[dotstyle=*](-6.,-1.3)\rput(-5.3,-1.3){$t_{f}^{2}$}
\psdots[dotstyle=*](-2.,-1.3)\rput(-1.3,-1.3){$t_{f+2}^{2}$}
\psdots[dotstyle=*](2.,-1.3)\rput(2.7,-1.3){$t_{f+4}^{2}$}
\psdots[dotstyle=*](6.,-1.3)\rput(6.7,-1.3){$t_{f+6}^{2}$}
\psdots[dotstyle=*](10.,-1.3)\rput(9.3,-1.3){$t_{f+8}^{2}$}
\psdots[dotstyle=*](-10.,-2.7)\rput(-9.3,-2.7){$t_{f-2}^{1}$}
\psdots[dotstyle=*](-6.,-2.7)\rput(-5.5,-2.7){$t_{f}^{1}$}
\psdots[dotstyle=*](-2.,-2.7)\rput(-1.3,-2.7){$t_{f+2}^{1}$}
\psdots[dotstyle=*](2.,-2.7)\rput(2.7,-2.7){$t_{f+4}^{1}$}
\psdots[dotstyle=*](6.,-2.7)\rput(6.7,-2.7){$t_{f+6}^{1}$}
\psdots[dotstyle=*](10.,-2.7)\rput(9.3,-2.7){$t_{f+8}^{1}$}
\psdots[dotstyle=*](-10.,-4.)\rput(-10.2,-4.5){$w_{f-2}$}
\psdots[dotstyle=*](-8.,-4.)\rput(-8.,-4.3){$r_{f-2}$}
\psdots[dotstyle=*](-6.,-4.)\rput(-6.3,-4.3){$w_{f}$}
\psdots[dotstyle=*](-4.,-4.)\rput(-4.,-4.3){$r_{f}$}
\psdots[dotstyle=*](-2.,-4.)\rput(-2.6,-4.3){$w_{f+2}$}
\psdots[dotstyle=*](0.,-4.)\rput(0.,-4.3){$r_{f+2}$}
\psdots[dotstyle=*](2.,-4.)\rput(2.6,-4.3){$w_{f+4}$}
\psdots[dotstyle=*](4.,-4.)\rput(4.,-4.3){$r_{f+4}$}
\psdots[dotstyle=*](6.,-4.)\rput(6.6,-4.3){$w_{f+6}$}
\psdots[dotstyle=*](8.,-4.)\rput(8.,-4.3){$r_{f+6}$}
\psdots[dotstyle=*](10.,-4.)\rput(9.3,-3.7){$w_{f+8}$}

\psdots[dotstyle=*](-7.6,-6.)\rput(-8.0,-6.2){$u_{-1}^{4}$}
\psdots[dotstyle=*](-4.52,-6.)\rput(-5.0,-6.2){$u_{1}^{4}$}
\psdots[dotstyle=*](-1.56,-6.)\rput(-1.0,-6.2){$u_{3}^{4}$}
\psdots[dotstyle=*](1.56,-6.)\rput(2.1,-6.2){$u_{5}^{4}$}
\psdots[dotstyle=*](4.52,-6.)\rput(5.1,-6.2){$u_{7}^{4}$}
\psdots[dotstyle=*](7.6,-6.)\rput(8.0,-6.2){$u_{9}^{4}$}

\psline[ArrowInside=->](-7.56,-6.)(-10.,-4.)
\psline[ArrowInside=->](-4.52,-6.)(-6.,-4.)
\psline[ArrowInside=->](-1.56,-6.)(-2.,-4.)
\psline[ArrowInside=->](1.56,-6.)(2.,-4.)
\psline[ArrowInside=->](4.52,-6.)(6.,-4.)
\psline[ArrowInside=->](7.6,-6.)(10.,-4.)

\psdots[dotstyle=*](-5.8,-7.3)\rput(-6.5,-7.3){$u_{-1}^{3}$}
\psdots[dotstyle=*](-3.5,-7.3)\rput(-4.1,-7.3){$u_{1}^{3}$}
\psdots[dotstyle=*](-1.2,-7.3)\rput(-1.7,-7.3){$u_{3}^{3}$}
\psdots[dotstyle=*](1.2,-7.3)\rput(1.7,-7.3){$u_{5}^{3}$}
\psdots[dotstyle=*](3.5,-7.3)\rput(4.0,-7.3){$u_{7}^{3}$}
\psdots[dotstyle=*](5.8,-7.3)\rput(6.5,-7.3){$u_{9}^{3}$}

\psline[ArrowInside=->](0.,-12.)(-2.48,-10.)
\psline[ArrowInside=->](0.,-12.)(-1.52,-10.)
\psline[ArrowInside=->](0.,-12.)(-0.5,-10.)
\psline[ArrowInside=->](0.,-12.)(0.5,-10.)
\psline[ArrowInside=->](0.,-12.)(1.52,-10.)
\psline[ArrowInside=->](0.,-12.)(2.48,-10.)

\psdots[dotstyle=*](-4.1,-8.7)\rput(-4.8,-8.7){$u_{-1}^{2}$}
\psdots[dotstyle=*](-2.48,-8.7)\rput(-3.1,-8.7){$u_{1}^{2}$}
\psdots[dotstyle=*](-.75,-8.7)\rput(-1.35,-8.7){$u_{3}^{2}$}
\psdots[dotstyle=*](0.85,-8.7)\rput(0.4,-8.7){$u_{5}^{2}$}
\psdots[dotstyle=*](2.5,-8.7)\rput(3.0,-8.7){$u_{7}^{2}$}
\psdots[dotstyle=*](4.1,-8.7)\rput(4.8,-8.7){$u_{9}^{2}$}
\psline[ArrowInside=->](-2.48,-10.)(-4.1,-8.7)
\psline[ArrowInside=->](-1.52,-10.)(-2.48,-8.7)
\psline[ArrowInside=->](-0.5,-10.)(-0.75,-8.7)
\psline[ArrowInside=->](0.5,-10.)(0.85,-8.7)
\psline[ArrowInside=->](1.52,-10.)(2.5,-8.7)
\psline[ArrowInside=->](2.48,-10.)(4.1,-8.7)
\psline[ArrowInside=->](-4.1,-8.7)(-5.8,-7.3)
\psline[ArrowInside=->](-2.48,-8.7)(-3.5,-7.3)
\psline[ArrowInside=->](-0.75,-8.7)(-1.2,-7.3)
\psline[ArrowInside=->](0.85,-8.7)(1.2,-7.3)
\psline[ArrowInside=->](2.5,-8.7)(3.5,-7.3)
\psline[ArrowInside=->](4.1,-8.7)(5.8,-7.3)
\psline[ArrowInside=->](-5.8,-7.3)(-7.6,-6.)
\psline[ArrowInside=->](-3.5,-7.3)(-4.52,-6.)
\psline[ArrowInside=->](-1.2,-7.3)(-1.56,-6.)
\psline[ArrowInside=->](1.2,-7.3)(1.56,-6.)
\psline[ArrowInside=->](3.5,-7.3)(4.52,-6.)
\psline[ArrowInside=->](5.8,-7.3)(7.6,-6.)
%

\psdots[dotstyle=*](-2.48,-10.)\rput(-3.0,-10){$u_{-1}^{1}$}
\psdots[dotstyle=*](-1.52,-10.)
\psdots[dotstyle=*](-0.5,-10.)
\psdots[dotstyle=*](0.5,-10.)
\psdots[dotstyle=*](1.52,-10.)
\psdots[dotstyle=*](2.48,-10.)\rput(3.1,-10){$u_{9}^{1}$}

\psline[linestyle=dotted](9,2)(10,2)
\psline[linestyle=dotted](-9,2)(-10,2)
\psline[linestyle=dotted](9,-8)(10,-8)
\psline[linestyle=dotted](-9,-8)(-10,-8)
\psline[linestyle=dotted](9,8)(10,8)
\psline[linestyle=dotted](-9,8)(-10,8)

\psline[ArrowInside=->](-8.,4.)(-6.,4.)
\psline[ArrowInside=->](-6.,4.)(-4.,4.)
\psline[ArrowInside=->](-4.,4.)(-2.,4.)
\psline[ArrowInside=->](-2.,4.)(0.,4.)
\psline[ArrowInside=->](0.,4.)(2.,4.)
\psline[ArrowInside=->](2.,4.)(4.,4.)
\psline[ArrowInside=->](4.,4.)(6.,4.)
\psline[ArrowInside=->](6.,4.)(8.,4.)

\psline[ArrowInside=->](-10.,0.)(-9.,0.)
\psline[ArrowInside=->](-9.,0.)(-8.,0.)
\psline[ArrowInside=->](-8.,0.)(-7.,0.)
\psline[ArrowInside=->](-7.,0.)(-6.,0.)
\psline[ArrowInside=->](-6.,0.)(-5.,0.)
\psline[ArrowInside=->](-5.,0.)(-4.,0.)
\psline[ArrowInside=->](-4.,0.)(-3.,0.)
\psline[ArrowInside=->](-3.,0.)(-2.,0.)
\psline[ArrowInside=->](-2.,0.)(-1.,0.)
\psline[ArrowInside=->](-1.,0.)(0.,0.)
\psline[ArrowInside=->](0.,0.)(1.,0.)
\psline[ArrowInside=->](1.,0.)(2.,0.)
\psline[ArrowInside=->](2.,0.)(3.,0.)
\psline[ArrowInside=->](3.,0.)(4.,0.)
\psline[ArrowInside=->](4.,0.)(5.,0.)
\psline[ArrowInside=->](5.,0.)(6.,0.)
\psline[ArrowInside=->](6.,0.)(7.,0.)
\psline[ArrowInside=->](7.,0.)(8.,0.)
\psline[ArrowInside=->](8.,0.)(9.,0.)
\psline[ArrowInside=->](9.,0.)(10.,0.)
\psline[ArrowInside=->](-10.,-4.)(-8.,-4.)
\psline[ArrowInside=->](-8.,-4.)(-6.,-4.)
\psline[ArrowInside=->](-6.,-4.)(-4.,-4.)
\psline[ArrowInside=->](-4.,-4.)(-2.,-4.)
\psline[ArrowInside=->](-2.,-4.)(0.,-4.)
\psline[ArrowInside=->](0.,-4.)(2.,-4.)
\psline[ArrowInside=->](2.,-4.)(4.,-4.)
\psline[ArrowInside=->](4.,-4.)(6.,-4.)
\psline[ArrowInside=->](6.,-4.)(8.,-4.)
\psline[ArrowInside=->](8.,-4.)(10.,-4.)
\psline[ArrowInside=->](-8.,4.)(-8.,2.7)
\psline[ArrowInside=->](-8,2.7)(-8,1.3)
\psline[ArrowInside=->](-8.,1.3)(-8.,0.)
\psline[ArrowInside=->](-4.,4.)(-4.,2.7)
\psline[ArrowInside=->](-4,2.7)(-4,1.3)
\psline[ArrowInside=->](-4.,1.3)(-4.,0.)
\psline[ArrowInside=->](0.,1.3)(0.,0.)
\psline[ArrowInside=->](0.,4.)(0.,2.7)
\psline[ArrowInside=->](0,2.7)(0,1.3)
\psline[ArrowInside=->](0.,1.3)(0.,0.)
\psline[ArrowInside=->](4.,4.)(4.,2.7)
\psline[ArrowInside=->](4,2.7)(4,1.3)
\psline[ArrowInside=->](4.,1.3)(4.,0.)
\psline[ArrowInside=->](8.,4.)(8.,2.7)
\psline[ArrowInside=->](8,2.7)(8,1.3)
\psline[ArrowInside=->](8.,1.3)(8.,0.)
\psline[ArrowInside=->](10.,-4.)(10.,-2.7)
\psline[ArrowInside=->](10,-2.7)(10,-1.3)
\psline[ArrowInside=->](10.,-1.3)(10.,0.)
\psline[ArrowInside=->](-10.,-4.)(-10.,-2.7)
\psline[ArrowInside=->](-10,-2.7)(-10,-1.3)
\psline[ArrowInside=->](-10.,-1.3)(-10.,0.)
\psline[ArrowInside=->](-6.,-4.)(-6.,-2.7)
\psline[ArrowInside=->](-6,-2.7)(-6,-1.3)
\psline[ArrowInside=->](-6.,-1.3)(-6.,0.)
\psline[ArrowInside=->](-2.,-1.3)(-2.,0.)
\psline[ArrowInside=->](-2,-2.7)(-2,-1.3)
\psline[ArrowInside=->](-2.,-4.)(-2.,-2.7)
\psline[ArrowInside=->](2.,-4.)(2.,-2.7)
\psline[ArrowInside=->](2,-2.7)(2,-1.3)
\psline[ArrowInside=->](2.,-1.3)(2.,0.)
\psline[ArrowInside=->](6.,-4.)(6.,-2.7)
\psline[ArrowInside=->](6,-2.7)(6,-1.3)
\psline[ArrowInside=->](6.,-1.3)(6.,0.)

\psline[ArrowInside=->](2.,12.)(-0.48,10.)
\psline[ArrowInside=->](2.,12.)(0.48,10.)
\psline[ArrowInside=->](2.,12.)(1.5,10.)
\psline[ArrowInside=->](2.,12.)(2.5,10.)
\psline[ArrowInside=->](2.,12.)(3.52,10.)

\psline[ArrowInside=->](-5.56,6.)(-8.,4.)
\psline[ArrowInside=->](-2.52,6.)(-4.,4.)
\psline[ArrowInside=->](0.44,6.)(0.,4.)
\psline[ArrowInside=->](3.56,6.)(4.,4.)
\psline[ArrowInside=->](6.52,6.)(8.,4.)
\end{scriptsize}

\rput(-4.,5.7){$F_{0}^-$}
\rput(-1,5.7){$F_{2}^-$}
\rput(+2.0,5.7){$F_{4}^-$}
\rput(+5.2,5.7){$F_{6}^-$}
\rput(-6,2){$F_{f-1}^-$}
\rput(-2,2){$F_{f+1}^-$}
\rput(+2,2){$F_{f+3}^-$}
\rput(+6,2){$F_{f+5}^-$}
\rput(-8,-2){$F_{f-2}^+$}
\rput(-4,-2){$F_{f}^+$}
\rput(0,-2){$F_{f+2}^+$}
\rput(+4,-2){$F_{f+4}^+$}
\rput(+8,-2){$F_{f+6}^+$}
\rput(-6.2,-6.3){$F_{-1}^+$}
\rput(-3.2,-6.3){$F_{1}^+$}
\rput(+0.0,-6.3){$F_{3}^+$}
\rput(+3.2,-6.3){$F_{5}^+$}
\rput(+6.2,-6.3){$F_{7}^+$}
\end{pspicture*}
\end{center}
  \caption{Face pairing polyhedron for $\mathcal{G}^{2/5}(n,f)$\label{fig:P25(n)}}
\end{figure}
\end{altproof}

We now consider the structures of the manifolds $M^{k/l}(n,f)$ of Theorem \ref{thm:Gkl(n,f)spine}.

\begin{theorem}\label{thm:commensurablehyperbolic}
Let $n\geq 2$, $k,l\geq 1$, $0\leq f<n$, and $fk\equiv 0\bmod n$, where $n,f$ are even, and suppose $(k,l)\in\{ (k,1), (1,l),$ $(5,2),(2,5) \}$. Then $M^{k/l}(n,f)$ is hyperbolic if and only if $M^{k/l}(n,0)$ is hyperbolic.
\end{theorem}

\begin{proof}
By \cite[Corollary 3.3]{ChinyereWilliamsFF} the shift $\theta_{G^{k/l}(n,0)}$ has order $n$, and hence the shift $\theta_{G^{k/l}(n,f)}$ also has order $n$. Suppose $M^{k/l}(n,0)$ is hyperbolic. Then since, by Theorem \ref{thm:Gkl(n,f)spine}, $M^{k/l}(n,0)$ is a closed, connected, orientable 3-manifold, $G^{k/l}(n,0)$ is a subgroup of $\mathrm{Isom}^+(H^3)\cong PSL(2,\mathbb{C})$. As in the proof of \cite[Theorem 3.1]{Maclachlan} (see also \cite[Theorem 3.1]{CRStopprop}, \cite[Theorem 3.1]{CRS}, \cite[Theorem 3.1]{BV}) Mostow rigidity implies that $E^{k/l}(n)$ is a subgroup of $PSL(2,\mathbb{C})$. Hence $G^{k/l}(n,f)$ is a subgroup of $PSL(2,\mathbb{C})$, and so $M^{k/l}(n,f)$ is hyperbolic. Repeating the argument with the roles of $G^{k/l}(n,0)$ and $G^{k/l}(n,f)$ interchanged and the roles of $M^{k/l}(n,0)$ and $M^{k/l}(n,f)$ interchanged proves the converse.
\end{proof}

\begin{remark}\label{rem:hyperboliccommensurable}
The argument of the proof of Theorem \ref{thm:commensurablehyperbolic} holds in the more general setting of \cite{BogleyShift}. Namely, if $\nu^{f_1},\nu^{f_2}:\pres{x,t}{t^n,W(x,t)}\rightarrow \pres{t}{t^n}$ ($n\geq 2$) are retractions given by $\nu^{f_1}(t)=\nu^{f_2}(t)=t$ and $\nu^{f_1}(x)=t^{f_1}, \nu^{f_2}(x)=t^{f_2}$ then $K_1=\mathrm{ker} (\nu^{f_1})$, $K_2=\mathrm{ker} (\nu^{f_2})$ have cyclic presentations $\mathcal{G}_n(\rho^{f_1}(W(x,t))), \mathcal{G}_n(\rho^{f_2}(W(x,t)))$, respectively (as defined in \cite[page 159]{BogleyShift}). Suppose that the shift automorphism has order $n$ for either (and hence both) of these groups, and that $K_1,K_2$ are fundamental groups of closed, connected, orientable $3$-manifolds $M_1,M_2$, respectively. Then $M_1$ is hyperbolic if and only if $M_2$ is hyperbolic. (Theorem \ref{thm:commensurablehyperbolic} corresponds to the case $W(x,t)=x^ltx^ktx^{-l}t^{-2}$, as in (\ref{eq:extension}), where $(k,l)\in\{ (k,1),(1,l),(5,2),(2,5) \}$.) We record that, since $K_1$ is finite if and only if $K_2$ is finite, it is also the case that $M_1$ is spherical if and only if $M_2$ is spherical.
\end{remark}

The significance of Theorem \ref{thm:commensurablehyperbolic} is that the manifolds $M^{k/l}(n,0)$ are known to be hyperbolic in many cases, as we now describe. By \cite[Corollary 2.1]{KimVesninFF}, if $n\geq 6$ is even then 
$M^{k/l}(n,0)$ is hyperbolic for all but finitely many pairs of coprime integers  $k,l$, and it is conjectured (\cite[page 658]{KimVesninFF}) that $n=6$ and $k=l=1$ is the only non-hyperbolic case (which is an affine Riemannian manifold by \cite[Proposition 6]{HKM}). Moreover, $M^{k/l}(n,0)$ is hyperbolic in each of the following cases: $k\geq 2,l=1$, $n\geq 6$ and even \cite[Theorem 3]{MaclachlanReid}; $k=1,l\geq 2$, $n\geq 6$ and even \cite[Corollary 3.5]{KimVesninFF}; $k=l=1$, $n\geq 8$ and even \cite[Theorem C]{HKM}. If $n=4$ and $k=1$ then $M^{k/l}(n,0)=M^{1/l}(4,0)$ is the lens space $L(4l^2+1,2l)$ and $G^{1/l}(4,0)\cong \Z_{4l^2+1}$ by \cite[Corollary 3.4]{KimVesninFF}. For $k\geq 2$ the manifolds $M^{k/1}(4,0)$ are not hyperbolic and are described in \cite[pages 170--171]{MaclachlanReid}. In the next theorem we consider the manifolds $M^{k/l}(4,f)$ where $k\geq 2$.

\begin{theorem}\label{thm:n=4}
Let $k\geq 2, l\geq 1$, where $\gcd(k,l)=1$ and suppose $fk\equiv 0 \bmod 4$. Then $G^{k/l}(4,f)$ contains a $\Z\oplus \Z$ subgroup and hence $M^{k/l}(4,f)$ is not hyperbolic.
\end{theorem}

\begin{proof}
We show that the subgroup $A$ of $E=E^{k/l}(4)$ generated by $x^k, tx^kt^{-1}$ is free abelian of rank 2. Then, since $\nu^f(x^k)=1$ and $\nu^f(tx^kt^{-1})=1$, $A=\Z\oplus\Z$ is a subgroup of $\mathrm{ker}(\nu^f)=G^{k/l}(4,f)$, and hence $M^{k/l}(4,f)$ is not hyperbolic by \cite[Theorem 3.3]{CallahanReid}.

Let $G=\mathrm{ker}(\nu^0)$, so $G=F^{k/l}(4)$ and is generated by $x_i=t^ixt^{-i}$, and let $G_1=\pres{x_0,x_2}{x_0^kx_2^k}$, $G_2=\pres{x_1,x_3}{x_1^kx_3^k}$, $H=\pres{a,b}{ab=ba}\cong \Z^2$. Let $\phi_1:H\rightarrow G_1$, $\phi_2:H\rightarrow G_2$, be given by $\phi_1(a)=x_0^k$, $\phi_1(b)=x_0^{-l}x_2^l$, $\phi_2(a)=x_3^{-l}x_1^l$, $\phi_2(b)=x_1^k$. We shall show that $\phi_1,\phi_2$ are injections. This implies that $G$ can be expressed as an amalgamated free product $G=G_1*_H G_2$, and $\phi_1(a)=x_0^k,\phi_1(b)=x_0^{-l}x_2^l$ generate a free abelian subgroup of $G$ (compare \cite[page 171]{MaclachlanReid}, which deals with the case $l=1$). That is, $x_0^k$ and $x_1^k$ generate a free abelian subgroup of $G$ and so $x^k$ and $tx^kt^{-1}$ generate a free abelian subgroup of $E$, as required.

We show that $\phi_1$ is an injection, the argument for $\phi_2$ being similar. Let $\alpha: G_1\rightarrow \Z_k$ be given by $\alpha(x_0)=1,\alpha(x_2)=1\in \Z_k$. Then $K=\mathrm{ker}(\alpha)$ is generated by $z=x_0^k$, $g_i=x_0^ix_2x_0^{-(i+1)}$ ($0\leq i\leq k-2$), $g_{k-1}=x_0^{k-1}x_2$, and has a presentation
\[ \pres{z,g_0,\ldots , g_{k-1}}{zg_0g_1\ldots g_{k-2}g_{k-1}, zg_1g_2\ldots g_{k-1}g_{0}, \ldots , zg_{k-1}g_0\ldots g_{k-3}g_{k-2}}. \]
Now let $\theta:K\rightarrow \Z^k$ be an epimorphism given by $\theta (g_1)=(1,0,\ldots ,0)$, $\theta (g_2)=(0,1,\ldots ,0), \ldots ,$ $\theta (g_{k-1})=(0,0,\ldots ,1)$, and $\theta(z)=(-1,-1,\ldots , -1)$. Let $l=qk+r$ where $q\geq 0$, $0\leq r<k$ and note that $r\neq 0$, since $\gcd(k,l)=1$. Then $x_0^{-l}x_2^l= (x_0^k)^{-q}x_0^{-r}(x_2^k)^qx_2^r=(x_0^k)^{-q}x_0^{-r}(x_0^{-k})^qx_2^r=z^{-2q-1} g_{k-r}g_{k-(r-1)}\ldots g_{k-1}$ and $x_2^k=z$. Therefore 
\begin{alignat*}{1}
\theta (x_0^{-l}x_2^l)&=(2q+1)(1,1,\ldots ,1)+(0,\ldots ,0,1,\ldots ,1)\\
&=(2q+1,\ldots ,2q+1,2q+2,\ldots ,2q+2)
\end{alignat*}
and $\theta (x_2^k)=(-1,-1,\ldots , -1)$ so $\theta (x_0^{-l}x_2^l), \theta (x_2^k)$ generate a free abelian subgroup of rank 2 in $\Z^k$. Moreover, in $G_1$, $x_0^{-l}x_2^l$ and $x_2^k$ commute and so generate a free abelian subgroup of $G_1$, and hence $\phi_1$ is injective.
\end{proof}

\section*{Acknowledgements}

I thank Jim Howie for insightful comments on a draft of this article and Alessia Cattabriga for interesting discussions relating to this work, and for bringing pertinent results from \cite{Minkus} and \cite{RuiniSpaggiari98} to my attention. I thank the referee for pointing out that \cite[Theorem 8]{McDermott} provides a family of cyclic presentations that are 3-manifold spines where the Whitehead graph is of type (I.5), and for other valuable suggestions that improved the paper.

\begin{Backmatter}

\printaddress

\end{Backmatter}
\end{document}